\documentclass[10pt]{amsart}
\usepackage{amsmath,amsthm,amsfonts,amssymb,amscd,flafter,pinlabel}
\pdfoutput=1
\usepackage[mathscr]{eucal}
\usepackage{graphics}
\usepackage{etex}
\usepackage{caption, subcaption}
\usepackage[usenames,dvipsnames]{color}
\usepackage{graphicx}

\usepackage{pgf}
\usepackage{tikz}
\usetikzlibrary{arrows,automata}
\usepackage[latin1]{inputenc}

\usepackage[colorlinks=true, urlcolor=NavyBlue, linkcolor=NavyBlue, citecolor=NavyBlue, pdftitle={Planar Legendrian theta graphs}, pdfauthor={Peter Lambert-Cole}, pdfsubject={Contact Topology}, pdfkeywords={Contact Topology, 57M27, 57R58}]{hyperref}

\newtheorem{theorem}{Theorem}[section]
\newtheorem{corollary}[theorem]{Corollary}
\newtheorem{conjecture}[theorem]{Conjecture}

\newtheorem{proposition}[theorem]{Proposition}

\newtheorem{lemma}[theorem]{Lemma}

\newtheorem{innercustomthm}{Theorem}
\newenvironment{customthm}[1]
  {\renewcommand\theinnercustomthm{#1}\innercustomthm}
  {\endinnercustomthm}

\theoremstyle{definition}
\newtheorem{definition}{Definition}[section]

\theoremstyle{definition}
\newtheorem{remark}[theorem]{Remark}
\newtheorem{observation}[theorem]{Observation}

\newcommand\ZZ{\mathbb{Z}}
\newcommand\RR{\mathbb{R}}

\newcommand\del{\partial}

\newcommand\tb{\text{tb}}

\newcommand\tw{\text{tw}}
\newcommand\rot{\text{rot}}
\newcommand\Rot{\text{Rot}}

\newcommand\Gstar{G^*}
\newcommand\Tstar{T^*}

\newcommand\tbbf{\overline{\tb}}
\newcommand\twbf{\overline{\tw}}
\newcommand\rotbf{\overline{\rot}}

\begin{document}

\title[{Planar Legendrian $\Theta$-graphs}]{Planar Legendrian $\Theta$-graphs}

\author[P. Lambert-Cole]{Peter Lambert-Cole}
\address{Department of Mathematics \\ Indiana University}
\email{pblamber@indiana.edu}
\urladdr{\href{https://pages.iu.edu/}{https://pages.iu.edu/~pblamber}}

\author[D. O'Donnol]{Danielle O'Donnol$^{\dag}$ }
\address{Department of Mathematics \\ Indiana University}
\email{odonnol@indiana.edu}
\urladdr{\href{http://pages.iu.edu/~odonnol/}{http://pages.iu.edu/~odonnol/}}

\thanks{$^{\dag}$ This work was partially supported by the National Science Foundation grant DMS-160036.}

\keywords{Contact Topology, Legendrian graphs}
\subjclass[2010]{53D10; 57M15; 05C10}

\begin{abstract}
We classify topologically trivial Legendrian $\Theta$-graphs and identify the complete family of nondestabilizeable Legendrian realizations in this topological class.  In contrast to all known results for Legendrian knots, this is an infinite family of graphs.  
We also show that any planar graph that contains a subdivision of a $\Theta$-graph or $S^1 \vee S^1$ as a subgraph will have an infinite number of distinct, topologically trivial nondestabilizeable Legendrian embeddings.  
Additionally, we introduce two moves, vertex stabilization and vertex twist, that change the Legendrian type within a smooth isotopy class.

\end{abstract}

\maketitle

\section{Introduction} 
\label{sec:introduction}

A {\it Legendrian graph} in a contact manifold $(M,\xi)$ is an embedding $g: G \rightarrow (M,\xi)$ of an abstract graph $G$ such that the image of each edge is tangent to the contact structure. In contrast with the study of Legendrian knots and links, relatively little is known about Legendrian graphs.  Only a few spatial graph types have been classified and there are several realization and non-realization results for Legendrian graphs in $(S^3,\xi_{std})$ or equivalently $(\RR^3,\xi_{std})$.

Eliashberg and Fraser proved that all Legendrian embeddings of an abstract tree are Legendrian isotopic \cite{Eliashberg_Fraser}.  The second author and Pavelescu classified topologically trivial Legendrian embeddings of the Lollipop and Handcuff graphs \cite{OP-1}.  A graph $G$ has a Legendrian embedding with all cycles maximal unknots if and only if it does not contain $K_4$ as a minor \cite{OP-1}.  However, $K_4$ does have a Legendrian embedding where each cycle has maximal Thurston-Bennequin number in its smooth isotopy class \cite{Tanaka}.  It is unknown whether every $G$ has such a Legendrian embedding.  The second author and Pavelescu solved the geography problem for Legendrian embeddings of the $\Theta$-graph where each cycle is unknotted \cite{OP-Theta}.  They gave necessary and sufficient conditions for fixed vectors to be realized as Thurston-Bennequin and rotation invariants of such Legendrian $\Theta$-graphs. In \cite{OP-3} they found further restrictions on the Thurston-Bennequin and rotation invariants of unknotted embeddings of complete and complete bipartite graphs.

In this paper, we completely classify topologically trivial Legendrian embeddings of the $\Theta$-graph in $(S^3,\xi_{std})$. Our key technical tool is the main result in \cite{LO-Planar} that the Legendrian ribbon $R_g$ and rotation invariant $\rot_g$ are a complete pair of invariants for topologically trivial Legendrian graphs.

Before stating the classification, we first describe some interesting consequences that distinguish Legendrian graphs from Legendrian knots.  For a fixed isotopy class of embeddings, the set of Legendrian realizations up to Legendrian isotopy is ``generated" by a set of nondestabilizeable Legendrian realizations.  Specifically, there exists a set $\mathcal{N}$ of Legendrian embeddings such that 
\begin{enumerate}
\item every Legendrian embedding in the isotopy class is obtained by a finite sequence of stabilizations from some element of $\mathcal{N}$, and 
\item no embedding in $\mathcal{N}$ admits a destabilization.
\end{enumerate}
See Section \ref{sec:background} for a precise description of stabilization.  For all smooth knot types for which the set $\mathcal{N}$ is known, it is finite.  However, this is not true for topologically trivial Legendrian $\Theta$-graphs.  There exists an infinite family $\{G_l\}$ of nondestabilizeable realizations.  The front projection of $G_l$ is given as follows.  Fix two points in the $xz$-plane and connect them by 3 arcs with no cusps.  Now, take $l \in \frac{1}{2}\ZZ_{\geq -\frac{1}{2}}$, collect the bottom 2 strands and add $l + \frac{1}{2}$ twists as in Figure \ref{fig:G-l}.  Call this graph $G_l$.

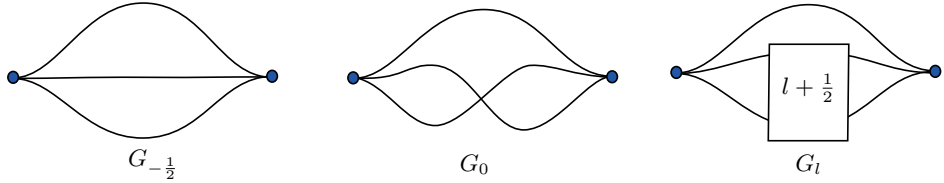
\begin{figure}[htpb!]
\centering
\labellist
	\small\hair 2pt
	\pinlabel $G_{-\frac{1}{2}}$ at 93 3
	\pinlabel $G_0$ at 290 3
	\pinlabel $G_l$ at 495 3
	\small\hair 4pt
	\pinlabel $l+\frac{1}{2}$ at 495 50

\endlabellist
\includegraphics[width=.8\textwidth]{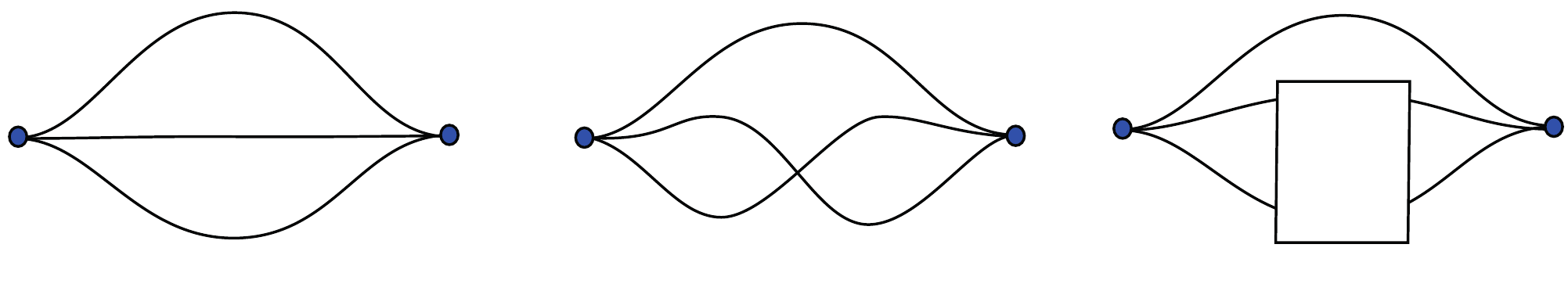}
\caption{On the left is a front projection of $G_{-\frac{1}{2}} $of the $\Theta-$graph.  In the center a $G_0$ realization of the $\Theta-$graph. 
On the right is a front projection of $G_l$ of the $\Theta-$graph.}
\label{fig:G_l}
\end{figure}

\begin{theorem}
\label{thrm:theta-nondestab}
The infinite family $\{G_l\}$, where $l \in \frac{1}{2} \ZZ$ and $l \geq -\frac{1}{2}$, are nondestabilizeable and pairwise non-Legendrian isotopic.
\end{theorem}

This result is a sharp contrast with what is known about Legendrian knots.  Moreover, it is always true if $G$ is abstractly planar and contains a subdivision of the $\Theta$-graph as a subgraph.

\begin{theorem}
\label{thrm:theta-infinite}
Let $G$ be an abstract planar graph that contains a subdivision of the $\Theta$-graph or $S^1 \vee S^1$ as a subgraph.  Then there exists infinitely many, pairwise-distinct, topologically trivial Legendrian embeddings $g: G \rightarrow (S^3,\xi_{std})$.
\end{theorem}

Fuchs and Tabachnikov \cite{Fuchs-Tabachnikov} proved that if $L_1,L_2$ are Legendrian knots that are smoothly isotopic, then there exists a sequence of (de)stabilizations that connect $L_1$ and $L_2$.  However, the stabilization move (which we refer to here as an {\it edge stabilization}) is not sufficient to connect any two Legendrian realizations of the same spatial graph.  

\begin{theorem}
\label{thrm:edge-stab-equivalence}
Two embeddings $G_l$ and $G_{l'}$ are related by a sequence of edge stabilizations if and only if $l - l' \in \ZZ$.
\end{theorem}

We introduce two new operations, called {\it vertex stabilization} and {\it vertex twist}, that modify Legendrian graphs without affecting the smooth isotopy class.  We conjecture that, together with edge stabilizations, these form a complete set of operations relating any two Legendrian graphs in the same smooth isotopy class.


\begin{conjecture}
Let $G$ be a graph and $g: G \rightarrow S^3$ a smooth embedding.  Then any two Legendrian realization $g_1$ and $g_2$ in $(S^3,\xi_{std})$ of $g$ are related, up to Legendrian isotopy, by a sequence of edge stabilizations, vertex stabilizations, and vertex twists.
\end{conjecture}

Using the techniques of \cite{LO-Planar}, we have a proof of this conjecture if $g$ is topologically trivial but not for arbitrary isotopy classes.

Now we address the classification.  The $\Theta$-graph consists of two vertices $v_1,v_2$ and three edges $e_1,e_2,e_3$.  There are three cycles $\gamma_1 = e_1 \cup e_2; \gamma_2 = e_2 \cup e_3; \gamma_3 = e_3 \cup e_1$.  We orient the cycle $\gamma_i$ so that $e_i$ is oriented from $v_1$ to $v_2$.  Throughout this paper, we let $\Theta$ denote the abstract graph and $\theta$ denote a topologically trivial embedding  $\theta: \Theta \rightarrow (S^3,\xi_{std})$.   The Thurston-Bennequin and rotation numbers of each cycle determine vectors $\tbbf_{\theta},\rotbf_{\theta} \in \ZZ^3$
\begin{align*}
\tbbf_{\theta} &:= (\tb_{\theta}(\gamma_1),\tb_{\theta}(\gamma_2),\tb_{\theta}(\gamma_3)) \\
\rotbf_{\theta} &:= (\rot_{\theta}(\gamma_1),\rot_{\theta}(\gamma_2),\rot_{\theta}(\gamma_3))
\end{align*}
We also think of the pair $(\tbbf_{\theta},\rotbf_{\theta})$ as a vector in $\ZZ^6$.  Let $\text{Rot}_{\theta} := \rot_{\theta}(\gamma_1) + \rot_{\theta}(\gamma_2) + \rot_{\theta}(\gamma_3)$ denote the {\it total rotation number} of $\theta$.  

Recall that for the study of knotted graphs, there is an important distinction between embeddings $g: G \rightarrow M$ and their images. Given an abstract graph $G$, a spatial embedding $j: G \rightarrow S^3$, and an automorphism $\phi$ of $G$, it is not true in general that there exists an ambient isotopy $h$ satisfying $h \circ j = j \circ \phi$.  Hence it is necessary to distinguish between embeddings of $G$ up to ambient isotopy and their images up to ambient isotopy.  Pictorally, this corresponds to a difference between unlabeled diagrams up to equivalence and labeled diagrams up to equivalence.  In particular, consider the graph $G_{-\frac{1}{2}}$ in Figure \ref{fig:G_l}.  All possible ways to label the edges and vertices have the same image; however, there exists different labeling choices such that there is no sequence of Reidemeister moves that equates the two as labeled graphs.  This is analagous to a non-invertible knot.

The automorphisms of $\Theta$ are permutations of the vertices and edges and so the automorphism group $\text{Aut}(\Theta)$ is $S_2 \times S_3$.  Under $\phi \in \text{Aut}(\Theta)$, the image or orientation of a cycle may change and so the invariants $\tbbf_{\theta},\rotbf_{\theta}$ change as well after relabeling by $\phi$.  In Section \ref{sec:classification}, we define representations $\rho_2,\rho_3$ of the symmetric groups $S_2,S_3$ on $\ZZ^6$ that describe the algebraic effect on the classical invariants $(\tbbf,\rotbf)$ of permuting the vertices and edges, respectively.  Certain symmetries of $(\tbbf,\rotbf)$ under the action by $\rho = \rho_2 \times \rho_3$ correspond to relabelings of a graph.

\begin{theorem}
The classification of topologically trivial Legendrian embeddings of the $\Theta$-graph in $(S^3,\xi_{std})$ is as follows:
\begin{enumerate}
\item The set $\{G_l\}$ for $l \in \frac{1}{2} \ZZ$ and $l \geq -\frac{1}{2}$ is the set of nondestabilizeable, topologically trivial Legendrian embeddings $\theta: \Theta \rightarrow (S^3,\xi_{std})$.
\item If $\tbbf,\rotbf \in \ZZ^3$ satisfy 
\begin{align*}
\tb(\gamma_i) + | \rot(\gamma_i)| &\leq -1 \\
\tb(\gamma_i) + \rot(\gamma_i) &= 1 \text{ mod }2 \\
\Rot & \in \{-1,0,1\}
\end{align*}
for $i = 1,2,3$, then there exists a topologically trivial Legendrian embedding $\theta$ with $(\tbbf_{\theta},\rotbf_{\theta}) = (\tbbf,\rotbf)$.
\item Let $(\tbbf,\rotbf)$ be an admissible pair.  Then
\begin{enumerate}
\item if $\Rot = 0$, there exist two distinct, topologically trivial embeddings $\theta_1,\theta_2$ with $(\tbbf_{\theta_i},\rotbf_{\theta_i}) = (\tbbf,\rotbf)$, and
\item if $\Rot = \pm 1$, there exists one topologically trivial embedding $\theta$ with $(\tbbf_{\theta},\rotbf_{\theta}) = (\tbbf,\rotbf)$.
\end{enumerate}
\item Let $(\tbbf,\rotbf)$ be an admissible pair with $\Rot = 0$.  Then
\begin{enumerate}
\item if $\rho_3(\tau) \cdot (\tbbf,\rotbf) = (\tbbf,\rotbf)$ for some transposition $\tau \in S_3$, then up to relabeling there exists one topologically trivial embedding $\theta$ with $(\tbbf_{\theta},\rotbf_{\theta}) = (\tbbf,\rotbf)$.
\item if $\rho_3(\tau) \cdot (\tbbf,\rotbf) \neq (\tbbf,\rotbf)$ for every transposition $\tau \in S_3$, then up to relabeling there exists two topologically trivial embeddings $\theta_1,\theta_2$ with $(\tbbf_{\theta_i},\rotbf_{\theta_i}) = (\tbbf,\rotbf)$.
\end{enumerate}
\end{enumerate}
\end{theorem}

\section{Background} 
\label{sec:background}

\subsection{Spatial graphs}

An {\it abstract graph} or {\it graph} $G$ is a 1-complex consisting of vertices and the edges connecting them.  
A {\it spatial graph} or {\it spatial embedding} is an embedding of a fixed abstract graph into a $3$-manifold $M$.  
Two spatial graphs $f(G)$ and $\bar{f}(G)$ are {\it ambient isotopic} if there exits an isotopy $h_t:\RR^3\to \RR^3$ such that $h_0=id$ and  $h_1(f(G))=\bar{f}(G).$  
As with knots, there is a set of Reidemeister moves for spatial graphs described by Kauffman in \cite{Kauffman}.

There is an additional subtlety when working with spatial graphs.  
It is necessary to distinguish between embeddings of a graph up to ambient isotopy and its images up to ambient isotopy.  
Given an abstract graph $G$, a spatial embedding $j: G \rightarrow S^3$, and an automorphism $\phi$ of $G$, it is not true in general that there exists an ambient isotopy $h$ satisfying $h \circ j = j \circ \phi$.  
In this article we are working with embeddings of graphs up to ambient isotopy.  
This corresponds to working with labeled diagrams up to equivalence.  
Without the labels the diagrams can look the same for different graphs because different edges are in the same positions.  
A spatial graph is {\it topologically trivial} if it lies on an embedded $S^2$ in $S^3$.  


\subsection{Legendrian graphs}

A (cooriented) {\it contact structure} $(M,\xi)$ on a 3-manifold $M$ is a plane field $\xi = \text{ker}(\alpha)$ where the 1-form $\alpha$ satisfies the nonintegrability condition $\alpha \wedge d \alpha > 0$.  As a result, the contact structure induces an orientation on $M$ and the 2-form $d \alpha$ orients the contact planes.  The basic example is the standard contact structure $\xi = \text{ker}(dz - y dx)$ on $\RR^3$.  This can be extended by the one-point compactification of $\RR^3$ to the standard structure on $S^3$.

A spatial graph is {\it Legendrian} if its image is tangent to the contact planes at every point.  
Two Legendrian graphs $G_0$ and $G_1$ are {\it Legendrian isotopic} if there exists a one-parameter family of Legendrian graphs $G_t$, where $t\in[0,1]$.  For a Legendrian graph in $(\mathbb{R}^3, \xi_{std})$, the projection to the $xz$-plane is called the {\it front projection}.  
Though we are working in $(S^3, \xi_{std})$ we will work with a projection in $\RR^3$; this is possible since Legendrian graphs and isotopes of Legendrian graphs can be perturbed slightly to miss a chosen point.  
The front projection is usually given as an immersion, since at any crossing the strand that is over is determined by the slope of the tangents.  
However, at times we will show diagrams where the crossings are indicated for added clarity.  
A front projection of the Legendrian graph that is in general position with all double points away from vertices is called {\it generic}.  
Two generic front projections of a Legendrian graph are related by the three Reidemeister moves for knots and links together with two moves given by the mutual position of vertices and edges \cite{Baader}.  
See Figure~\ref{figure:Rmoves}.


\begin{figure}[htpb!]
\begin{center}
\begin{picture}(389, 170)
\put(0,0){\includegraphics[width=5.5in]{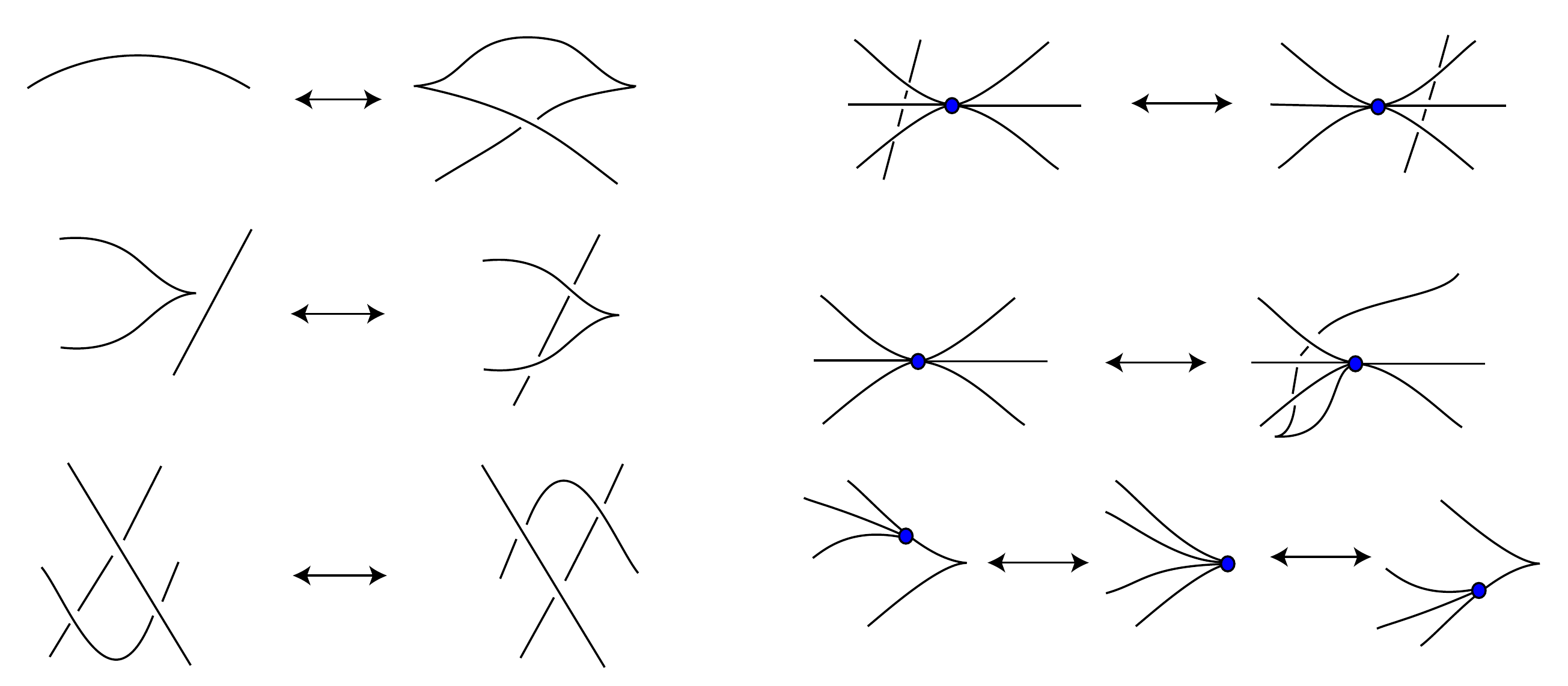}}
\put(84, 153){I}
\put(82,99){II}
\put(80,33){III}
\put(293, 153){III$_v$}
\put(290,88){V}
\put(260,35){V}
\put(330,37){V}
\end{picture}
\caption{\small Legendrian isotopy moves for graphs in the front projection:  Reidemeister moves I, II, and III, an edge passing under or over a vertex (III$_v$), an edge adjacent to a vertex rotates to the other side of the vertex (V). Move V is shown with edges on both sides and on one side of the vertex.  Reflections of these moves that are Legendrian front projections are also allowed.}
\label{figure:Rmoves}
\end{center}
\end{figure}
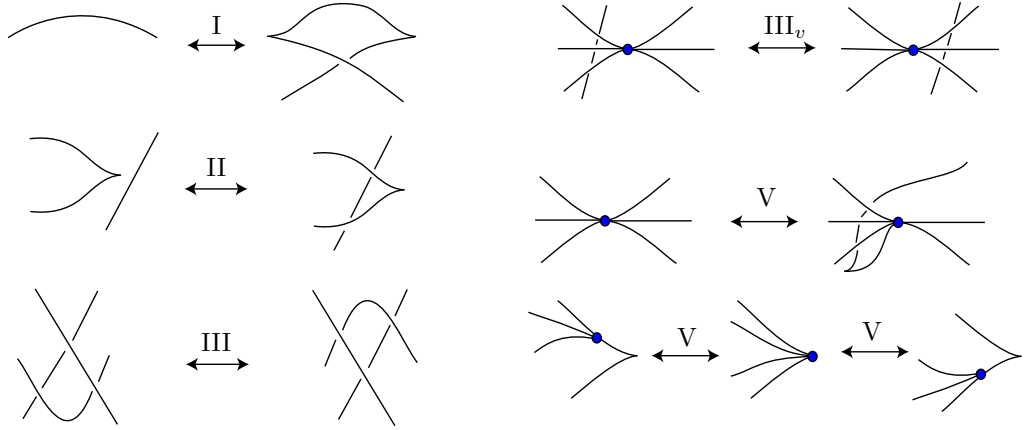

There are a number of different invariants for Legendrian graphs.  
The first two are a generalization of the classical invariants, Thurston--Bennequin, $\tb$, and rotation number, $\rot$, to graphs \cite{OP-1}.  
For a Legendrian graph $g$, with an ordering on its cycles, the {\it Thurston--Bennequin number}, $\overline{\tb_g}$, is the ordered list of the Thurston--Bennequin numbers for its cycles.
Similarly, the {\it rotation number}, $\overline{\rot_g}$, is the ordered list of the rotation numbers for its cycles.  

In a front diagram, the Thurston-Bennequin number and rotation number of an oriented cycle $\gamma$ can be computed by the formulas
\begin{align*}
\tb_g(\gamma) &= -\frac{1}{2} \# \text{cusps} + \text{ writhe } &
\rot_g(\gamma) &= \frac{1}{2} \left( \# \text{ down cusps } - \# \text{ up cusps} \right)
\end{align*}

Next, a {\it Legendrian ribbon} $R_g$ for a Legendrian graph $g: G \rightarrow (M,\xi)$ is a compact, oriented surface such that
\begin{enumerate}
\item $R_g$ contains $g(G)$ as its 1-skeleton,
\item $\xi$ has no negative tangencies to $R_g$,
\item there exists a vector field $X$ on $R_g$ tangent to the characteristic foliation of $R_g$ whose time-$t$ flow $\phi_t$ satisfies $\cap_{t \geq 0} \phi_t(R_g) = g(G)$,
\item the oriented boundary of $R_g$ is positively transverse to the contact structure $\xi$.
\end{enumerate}
The Legendrian ribbon is unique up to ambient contact isotopy and thus an invariant of $g$.  
This gives rise to a couple more invariants.  
The underlying unoriented surface $\overline{R}_g$ is the {\it contact framing} of $g$. 
The transverse link that forms the boundary of $R_g$ is the {\it transverse pushoff} of $g$.  

A topological knot type $K$ is Legendrian simple if $(\tb,\rot)$ are a complete set of invariants for Legendrian realizations of $K$.  
In other words, there is at most one Legendrian realization of $K$ for each choice of $(\tb,\rot)$.  
Several topological knot types are known to be Legendrian simple.  
A few examples are the unknot, the figure-8 knot, and torus knots \cite{Eliashberg_Fraser,Etnyre_Honda,Ding-Geiges}.  
For Legendrian graphs much less is known about Legendrian simplicity.  
In the case of topologically trivial graphs, we have the following result:

\begin{theorem}[\cite{LO-Planar}]
\label{thrm:complete}
Let $G$ be a trivalent, planar graph.  Then $(R_g,\rot_g)$ is a complete set of invariants for topologically trivial Legendrian embeddings $g: G \rightarrow (S^3,\xi_{std})$
\end{theorem}

For a Legendrian knot $K$, different Legendrian knots in the same topological class can be obtained by stabilizations, where a single twist is added in an arc of the knot, consistent with the contact structure.  
This can be defined in $(\mathbb{R}^3,\xi_{std})$ with the front projection.   
A \textit{stabilization} means replacing an arc in the front projection of $K$ by one of the zig-zags in Figure~\ref{fig:stabilizations}. 
The stabilization is positive if the new cusps are oriented downwards and negative if the cusps are oriented upwards 
The classical invariants of the stabilized knot satisfy $\tb(S_{\pm}(K)) = \tb(K) -1$ and $\rot(S_{\pm}(K)) = \rot(K) \pm 1$. Thus stabilization changes the Legendrian isotopy class.  It is a well defined operation and can be performed anywhere along a knot.


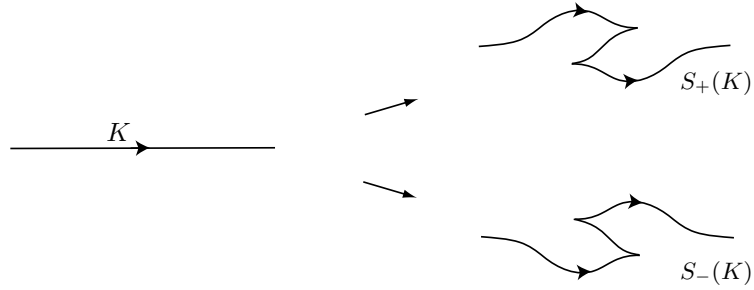
\begin{figure}[htpb!]
\begin{center}
\begin{picture}(300, 120)
\put(0,0){\includegraphics[width=4.3in]{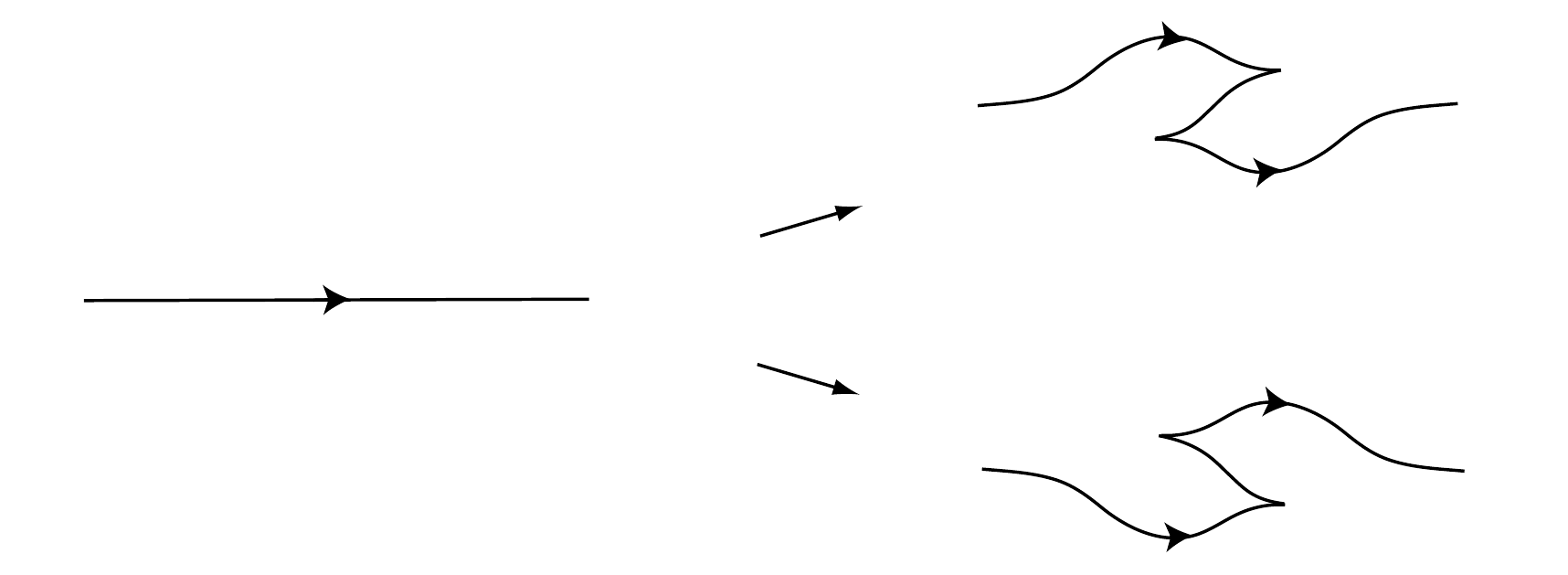}}
\put(53,57){$K$}
\put(270,77){\small $S_+(K)$}
\put(270,5){\small $S_-(K)$}
\end{picture}
\caption{Positive and negative stabilizations in the front projection of a knot or edge }\label{fig:stabilizations}
\end{center}
\end{figure}

For Legendrian graphs, {\it edge stabilizations} along a fixed edge are defined in the same way.  They can be performed at any point along a given edge but stabilizing along different edges will not be equivalent in general.

Consider the map diffeomorphism $\phi: (x,y,z) \mapsto (x,-y,-z)$ of $\RR^3$ to itself.  It is an orientation-preserving contactomorphism of $\RR^3$ since $\phi^*(\alpha) = - \alpha$.  
However, it reverses the coorientation of the contact planes.  If $g$ is a Legendrian graph, the {\it mirror} $M(g)$ is the image of $g$ under the map $\phi$.  
In the front projection, this corresponds to mirroring across the x-axis.

\begin{lemma}
Let $g: G \rightarrow (S^3,\xi_{std})$ be a Legendrian graph and let $M(g)$ be its mirror.  Then
\begin{enumerate}
\item the embeddings $g$ and $M(g)$ are ambient isotopic,
\item the contact framings $\overline{R}_g,\overline{R}_{M(g)}$ agree,
\item the Legendrian ribbons $R_g,R_{M(g)}$ have opposite coorientations,
\item if $\gamma$ is an oriented cycle of $G$, then 
\[\tb_{M(g)}(\gamma) = \tb_g(\gamma) \qquad \rot_{M(g)}(\gamma) = - \rot_g(\gamma)\]
\end{enumerate}
\end{lemma}

\begin{proof}
The map $\phi$ is a rotation by $\pi$ around the $x$-axis, so $g$ and $M(g)$ are clearly smoothly isotopic.  In addition, this rotation takes $\overline{R}_g$ to $\overline{R}_{M(g)}$ but the coorientations on the contact framing induced by $\xi_{std}$ differ by a sign.  Finally, the rotation does not change the number of cusps or the writhe of a cycle but it does exchange up and down cusps.  Thus, $\tb_g(\gamma)$ is fixed but $\rot_g(\gamma)$ changes sign.
\end{proof}

\section{Classification of $\Theta$-graphs} 
\label{sec:classification}

The $\Theta$-graph consists of two vertices $v_1,v_2$ and three edges $e_1,e_2,e_3$.  There are three cycles $\gamma_1 = e_1 \cup e_2; \gamma_2 = e_2 \cup e_3; \gamma_3 = e_3 \cup e_1$.  We orient the cycle $\gamma_i$ so that $e_i$ is oriented from $v_1$ to $v_2$.  Let $\theta: \Theta \rightarrow (S^3,\xi_{std})$ be a Legendrian embedding.  The Thurston-Bennequin and rotation numbers of each cycle determine vectors $\tbbf_{\theta},\rotbf_{\theta} \in \ZZ^3$
\begin{align*}
\tbbf_{\theta} &:= (\tb_{\theta}(\gamma_1),\tb_{\theta}(\gamma_2),\tb_{\theta}(\gamma_3)) \\
\rotbf_{\theta} &:= (\rot_{\theta}(\gamma_1),\rot_{\theta}(\gamma_2),\rot_{\theta}(\gamma_3))
\end{align*}

Let $\text{Rot}_{\theta} := \rot_{\theta}(\gamma_1) + \rot_{\theta}(\gamma_2) + \rot_{\theta}(\gamma_3)$ denote the {\it total rotation number} of $\theta$.  It is also convenient to think of the Thurston-Bennequin invariant as a triple of half-integers assigned to each edge.  Define
\[tw_{\theta}(e_i) = \frac{1}{2}\left( \tb_{\theta}(\gamma_{i-1}) + \tb_{\theta}(\gamma_{i}) - \tb_{\theta}(\gamma_{i+1}) \right)\]
for $i = 1,2,3$ mod $3$.  It follows that $\tb_{\theta}(\gamma_i) = \tw_{\theta}(e_i) + \tw_{\theta}(e_{i+1})$.

\begin{lemma}
\label{lemma:tb-determines-contact}
Suppose that $\theta,\theta'$ are topologically trivial Legendrian embeddings of $\Theta$ such that $\tbbf_{\theta} = \tbbf_{\theta'}$.  Then $\overline{R}_{\theta}$ and $\overline{R}_{\theta'}$ are isotopic.
\end{lemma}

\begin{proof}
Let $R_{\theta},R_{\theta'}$ be the Legendrian ribbons of $\theta,\theta'$ and let $\Sigma_{\theta},\Sigma_{\theta'}$ be a smooth spheres containing the images of $\theta,\theta'$, respectively.

The $\Theta$-graph is trivalent, so we can isotope $R_{\theta}$ and $\Sigma_{\theta}$ rel $\theta(\Theta)$ so that they agree in open neighborhoods of $\theta(v_1),\theta(v_2)$.  Up to isotopy, the Legendrian ribbon $R_{\theta}$ can be obtained from a tubular neighborhood of $\theta(\Theta)$ in $\Sigma_{\theta}$ by adding $\tw_{\theta}(e_i)$ half twists along the edge $e_i$ for $i=1,2,3$.  $R_{\theta'}$ can be obtained from $\Sigma_{\theta'}$ similarly.  Since $\tbbf_{\theta} = \tbbf_{\theta'}$, this implies that the underlying unoriented surfaces $\overline{R}_{\theta}$ and $\overline{R}_{\theta'}$ are isotopic.
\end{proof}

The contact planes $\xi_{\theta(v_1)}$ and $\xi_{\theta(v_2)}$ are oriented.  Define the sign of a vertex to be
\[\sigma_{\theta}(v) := \begin{cases}
1 & \text{if $\theta(e_1),\theta(e_2),\theta(e_3)$ are positively cyclically ordered in $\xi_{\theta(v)}$} \\
-1 & \text{if $\theta(e_1),\theta(e_2),\theta(e_3)$ are negatively cyclically ordered in $\xi_{\theta(v)}$} \\
\end{cases} \]

\begin{lemma}
\label{lemma:total-rot}
Let $\theta$ be a topologically trivial Legendrian embedding with rotation invariant $\rot_{\theta}$.  Then the total rotation number satisfies
\[\Rot_{\theta} = \frac{1}{2} \left( \sigma_{\theta}(v_1) - \sigma_{\theta}(v_2) \right) \]
\end{lemma}

\begin{proof}
This is a reformulation of Lemma 2.1 in \cite{LO-Planar} and Lemma 5 in \cite{OP-Theta}.
\end{proof}

For each Legendrian unknot $K$, the Thurston-Bennequin and rotation numbers must satisfy
\begin{align}
\tb(K) + |\rot(K)| &\leq -1\\
\tb(K) + \rot(K) &= 1 \text{ mod }2
\end{align}
Each cycle of a topologically trivial Legendrian $\theta$ must also satisfy these constraints.  In addition, Lemma \ref{lemma:total-rot} implies that the total rotation number must satisfy
\begin{align}
\Rot_{\theta} &\in \{-1,0,1\}
\end{align}
A pair $(\tbbf,\rotbf)$ is {\it admissible} if it satisfies the above three restrictions.

Recall the set of graphs $\{G_l\}$ defined in Section \ref{sec:introduction}.  By abuse of notation, we let $G_l$ denote any embedding $\theta$ whose image is $G_l$.  Up to isotopy, there are $12$ possible labelings: there are $2!$ ways to label the vertices and $3!$ ways to label the edges.

\begin{lemma}
\label{lemma:Gl-basics}
Fix $l \in \frac{1}{2}\ZZ_{\geq -\frac{1}{2}}$ and let $\theta$ be an embedding with image $G_l$.  Then
\begin{enumerate}
\item if $\theta(e_1)$ is the top edge, then
\[ \tbbf_{\theta} = (-1,-2 - 2l,-1) \qquad \text{ and } \qquad \twbf_{\theta} = (l, -1 - l, -1 - l),\]
\item the total rotation number satisfies
\[\Rot_{\theta} = 2l + 1 \text{ mod }2, \]
\item $\theta$ is topologically trivial, and
\item $\theta$ is nondestabilizeable.
\end{enumerate}
\end{lemma}

\begin{proof}
To prove (1), note that pairing the top edge with either of the remaining edges gives a cycle with 2 cusps and no crossings.  The cycle composed of the bottom two edges has 2 cusps and writhe $-1 - 2l$.  The edge twisting $\twbf_{\theta}$ can now be immediately calculated from $\tbbf_{\theta}$.

For (2), label the left vertex $v_1$ and right vertex $v_2$ and label the edges of $G_l$ at the left vertex in ascending order.  This implies that $\sigma_{\theta}(v_1) = 1$.  If $2l = 1 \text{ mod } 2$, then the edges are labeled in ascending order at the right vertex and so $\sigma_{\theta}(v_2) = 1$ as well.  If $2l = 0 \text{ mod }2$, then $\sigma_{\theta}(v_2) = -1$.  Thus, Lemma \ref{lemma:total-rot} implies the statement for this labeling of $G_l$.  Transposing the labels on the vertices or on a pair of vertices modifies $\Rot_{\theta}$ by multiplication by $-1$ and therefore does not change $\Rot_{\theta} \text{ mod }2$.

For (3), the initial graph $G_{-\frac{1}{2}}$ is clearly topologically trivial.  Moreover, twisting along the bottom $2$ strands does not change the ambient isotopy class.  Finally, for (4), each $G_l$ is nondestabilizeable since each edge is contained in a cycle with $\tb = -1$.
\end{proof}

\begin{lemma}
\label{lemma:realization}
For each admissible $(\tbbf,\rotbf)$, there exists a stabilization of some $G_l$ to an embedding $\theta$ such that $(\tbbf_{\theta},\rotbf_{\theta}) = (\tbbf, \rotbf)$.
\end{lemma}

\begin{proof}
To simplify notation, define $r_i = \rot(\gamma_i)$ for $i = 1,2,3$.  The vector $\tbbf$ determines a vector $\overline{\tw} = (t_1,t_2,t_3)$ as described above.  Up to cyclic relabeling of the edges, we can assume that $t_1 \geq t_2,t_3$.  For $i=1,2,3$, define
\begin{align*}
a_2 &:= \frac{1}{2} \left( -1 - \tb(\gamma_1) - r_{1} \right) & b_2 &:= \frac{1}{2} \left( -1 -\tb(\gamma_1) + r_{1} \right) \\
 &= \frac{1}{2} \left( -1 - t_1 - t_2 - r_{1} \right) &  &= \frac{1}{2} \left( -1 -t_1 - t_2 + r_{1} \right) \\
a_3 &:= \frac{1}{2} \left( -1 - \tb(\gamma_3) + r_{3} \right) & b_3 &:= \frac{1}{2} \left( -1 - \tb(\gamma_3) - r_{3} \right) \\
 &= \frac{1}{2} \left( -1 - t_1 - t_3 + r_{3} \right) &  &= \frac{1}{2} \left( -1 - t_1 - t_3 - r_{3} \right) \\
a_1 &:= \text{min}(a_2,a_3) & b_1 &:= \text{min}(b_2,b_3)
\end{align*}
Admissibility implies that $a_i,b_i$ are nonnegative integers.

First, suppose that $t_1 \geq -\frac{1}{2}$.  Set $l = t_1$ and let $G = G_l$ be a nondestabilizeable Legendrian realization of $\Theta$ with $\tw_G(e_1) = l$.  As a result, by Lemma \ref{lemma:Gl-basics}, $\tw_G(e_2) = \tw_G(e_3) = -1 - t_1$ and furthermore $\rot_G(\gamma_{1}) = \rot_G(\gamma_{3}) = 0$ since $\tb_G(\gamma_1) = \tb_G(\gamma_3) = -1$.  Apply $a_i$ positive stabilizations and $b_i$ negative stabilizations to the oriented edge $e_i$ for $i=2,3$.  A simple computation shows that the resulting graph $\theta$ satisfies $\tbbf_{\theta} = \tbbf$ and $\rotbf_{\theta} = \rotbf$ and is the required graph.

Secondly, suppose that $t_1 < -\frac{1}{2}$.  We claim that $t_1 + a_1 + b_1 \geq -\frac{1}{2}$.  There are four cases to prove this, depending on the value of $a_1,b_1$.

If $a_1 = a_2$ and $b_1 = b_2$, then
\[t_1 + a_1 + b_1 = t_1 + \frac{1}{2} \left( -2 - 2 t_1 - 2 t_2 \right) = -1 - t_2 > -\frac{1}{2}\]
since $t_i \leq t_1 < -\frac{1}{2}$.  A similar argument holds if $a_1 = a_3$ and $b_1 = b_3$.

If $a_1 = a_2$ and $b_1 = b_3$, then

\begin{align*}
2(a_3 - a_2) &= t_2 - t_3 + r_1 + r_3 \geq 0 \\
2(b_2 - b_3) &= t_3 - t_2 + r_1 + r_3  \geq 0
\end{align*}
which implies that $r_1 + r_3 = t_2 - t_3 = 0$. So
\[t_1 + a_1 + b_1 = \frac{1}{2}\left (-2 -t_2 - t_3 \right)\]
In addition, $t_2 = t_2$ and so $t_1 + a_1 + b_1 > 0$.  A similar argument holds if $a_1 = a_3$ and $b_1 = b_2$.

Now, set $l = t_1 + a_1 + b_1$ and let $G = G_l$ such that $\tw_G(e_1) = l$.  Apply $a_1$ positive and $b_1$ negative stabilizations to $e_1$ and then apply $a_i - a_1$ positive and $b_i - b_1$ negative stabilizations to $e_i$ for $i=2,3$.  Again, a simple computation shows that the resulting graph $\theta$ satisfies $\tbbf_{\theta} = \tbbf$ and $\rotbf_{\theta} = \rotbf$ and is the required graph.
\end{proof}

\begin{lemma}
\label{lemma:realization-rot0}
Let $(\tbbf,\rotbf)$ be admissible.  If $\Rot = 0$, then there exists two inequivalent Legendrian embeddings $\theta_+,\theta_-$ satisfying $\tbbf = \tbbf_{\theta_+} = \tbbf_{\theta_-}$ and $\rotbf = \rotbf_{\theta_+} = \rotbf_{\theta_-}$.
\end{lemma}

\begin{proof}
By Lemma \ref{lemma:realization}, there exists some Legendrian embedding $\theta_+$ with $\tbbf_{\theta_+} = \tbbf$ and $\rotbf_{\theta_+} = \rotbf$. It is obtained from some $G_l$ by performing $p_i$ positive and $n_i$ negative stabilizations along the edge $e_i$ for $i = 1,2,3$.  See Figure \ref{fig:theta+-}.  Let $h$ be the graph obtained by instead performing $n_i$ positive and $p_i$ negative stabilizations along the edge $e_i$ for $i=1,2,3$.  Consequently, $\tbbf_{\theta_+} = \tbbf_h$ and since $\Rot_{\theta_+} = 0$, this also implies that $\rotbf_h = -\rotbf_{\theta_+}$.  Let $\theta_- = M(h)$ be the mirror of $h$.  Mirroring preserves $\tbbf$ and switches the sign of the rotation number of each cycle.  Thus $\tbbf_{\theta_-} = \tbbf_{\theta_+}$ and $\rotbf_{\theta_-} = \rotbf_{\theta_+}$.

We can perform these stabilizations away from the vertices, so we can assume that $R_{\theta_+}$ and $R_h$ agree near the vertices and thus have the same coorientation.  In addition, since $\tbbf_{\theta_+} = \tbbf_h$ and both $\theta_+$ and $h$ are topologically trivial, there is a smooth isotopy of $\theta_+$ to $h$ taking $R_{\theta_+}$ to $R_h$.  The ribbon $R_{\theta_-}$ is smoothly isotopic to $R_h$ but has the opposite coorientation.  Thus, $\theta_+$ and $\theta_-$ are distinguished by their Legendrian ribbons.
\end{proof}


\begin{figure}[htpb!]
\begin{center}
\begin{picture}(400, 200)
\put(0,0){\includegraphics[width=5.5in]{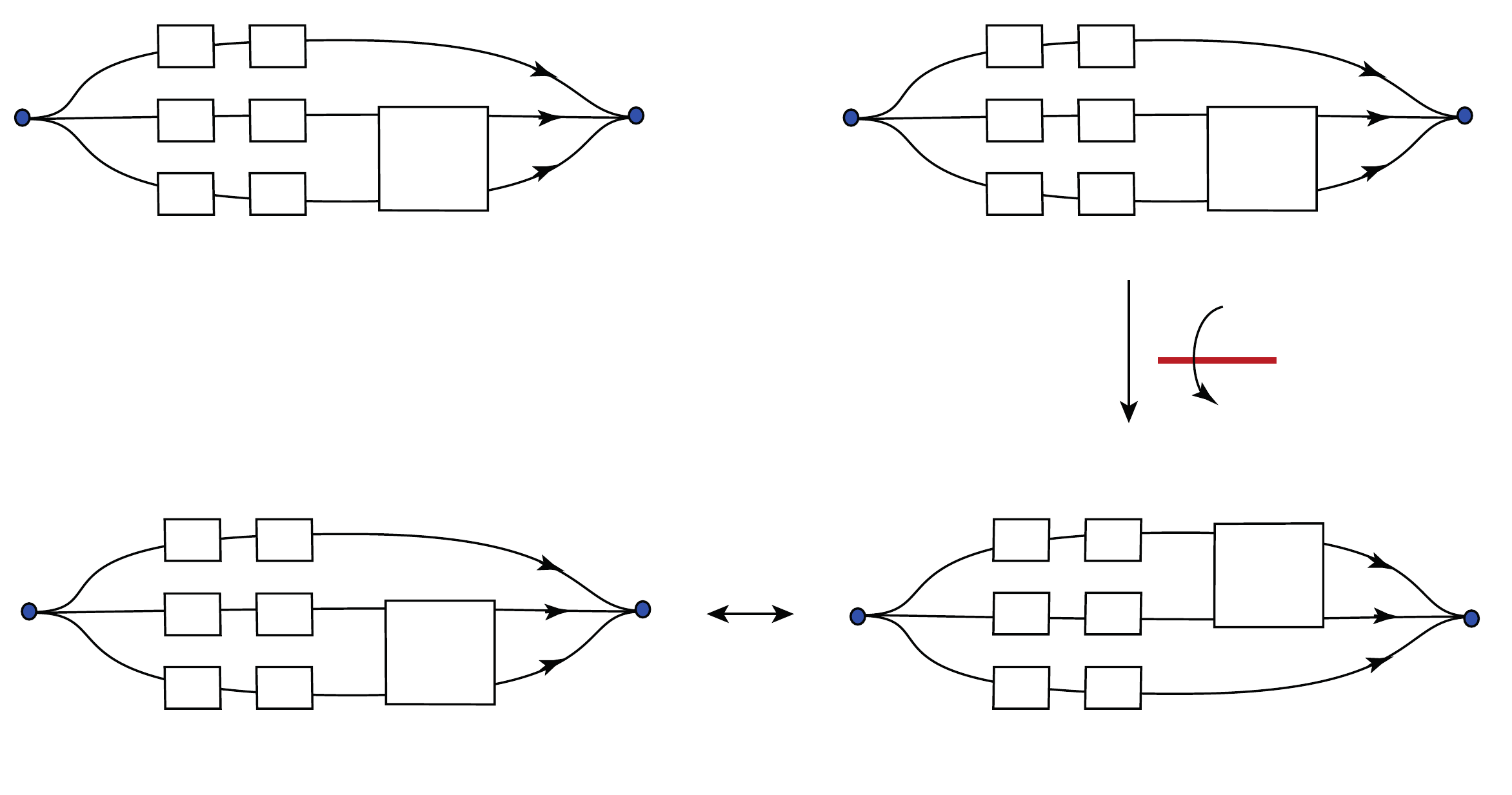}}
\put(45,195){\small $p_1$}
\put(45,175){\small $p_2$}
\put(45,157){\small $p_3$}
\put(68,195){\small $n_1$}
\put(68,175){\small $n_2$}
\put(68,156){\small $n_3$}
\put(105,166){\small $l+\frac{1}{2}$}
\put(80,138){\large $\theta_+$}

\put(264,195){\small $n_1$}
\put(264,175){\small $n_2$}
\put(264,156){\small $n_3$}
\put(287,195){\small $p_1$}
\put(287,175){\small $p_2$}
\put(287,157){\small $p_3$}
\put(322,166){\small $l+\frac{1}{2}$}
\put(300,142){\large $h$}

\put(46,65){\small $p_1$}
\put(45,26){\small $p_2$}
\put(45,45){\small $p_3$}
\put(70,65){\small $n_1$}
\put(70,26){\small $n_2$}
\put(70,45){\small $n_3$}
\put(105,36){\small $l+\frac{1}{2}$}
\put(290,6){\large $M(h)$}

\put(264,65){\small $p_3$}
\put(264,26){\small $p_1$}
\put(264,45){\small $p_2$}
\put(288,65){\small $n_3$}
\put(288,26){\small $n_1$}
\put(288,45){\small $n_2$}
\put(325,56){\small $l+\frac{1}{2}$}
\put(80,6){\large $\theta_-$}
\end{picture}
\caption{The graph $\theta_+$ (top left) is obtained from $G_l$ by applying $p_i$ positive and $n_i$ negative stabilizations along $e_i$.  Our convention is that the left box denote positive stabilizations and the right box denotes negative stabilizations.  See Figure \ref{fig:StabEdge}. The graph $h$ (top right) is obtained from $G_l$ by instead applying $n_i$ positive and $p_i$ negative stabilizations.  The mirror of $h$ (bottom right) is obtained by rotation around the $x$-axis.  Note that a positive zigzag in $h$ becomes a negative zigzag in $M(h)$.  These commute so we can move the $n_i$ positive stabilizations to the left by a Legendrian isotopy.  By a Legendrian isotopy, we can move the bottom strand of $M(h)$ to the top and obtain $\theta_-$ (bottom left).  If $\Rot_{\theta_+} = 0$, then $\theta_+$ and $\theta_-$ have the same invariants $\tbbf,\rotbf$.}\label{fig:theta+-}
\end{center}
\end{figure}
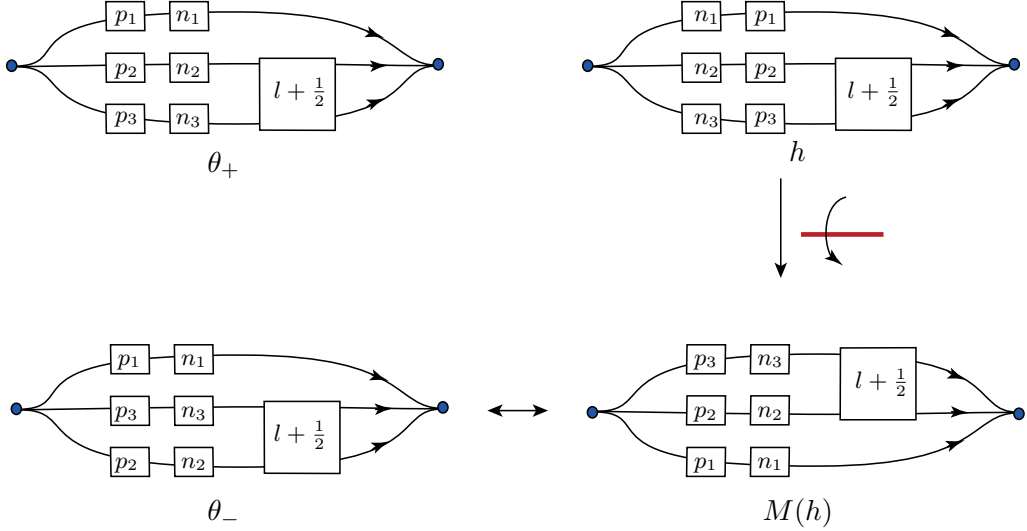


\begin{figure}[htpb!]
\begin{center}
\begin{picture}(350, 60)
\put(0,0){\includegraphics[width=4.5in]{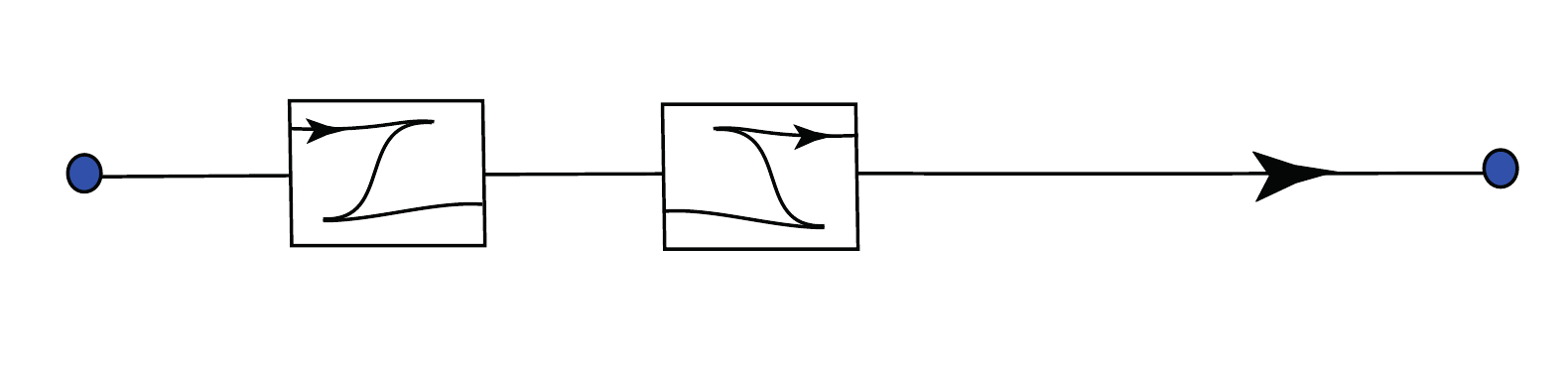}}
\put(65,16){\small positive}
\put(55,6){\small stabilizations}

\put(142,16){\small negative}
\put(135,6){\small stabilizations}
\end{picture}
\caption{In Figure \ref{fig:theta+-}, our convention is to let the left box denote positive stabilizations and the right box denote negative stabilizations.  }\label{fig:StabEdge}
\end{center}
\end{figure}
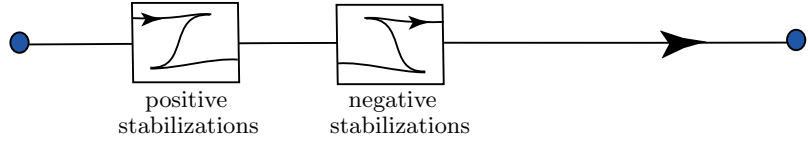

\begin{lemma}
\label{lemma:realization-rot1}
Let $\theta_1,\theta_2: \Theta \rightarrow (S^3,\xi_{std})$ be topologically trivial with the same $\tbbf,\rotbf$ and such that $\Rot_{\theta_i} = \pm 1$.  Then $\theta_1$ and $\theta_2$ are Legendrian isotopic.
\end{lemma}

\begin{proof}
If $\tbbf_{\theta_1} = \tbbf_{\theta_2}$, then by Lemma \ref{lemma:tb-determines-contact} the contact framings $\overline{R}_{\theta_1}$ and $\overline{R}_{\theta_2}$ agree.  Since $\rotbf_{\theta_1} = \rotbf_{\theta_2}$, Theorem \ref{thrm:complete} implies that only the coorientation on the contact framing distinguishes the two embeddings.

By Lemma \ref{lemma:total-rot}
\[\Rot_{\theta_1} = \frac{1}{2}\left(\sigma_{\theta_1}(v_1) - \sigma_{\theta_1}(v_2) \right) = \pm 1\]
However, if $R_{\theta_1}$ and $R_{\theta_2}$ have opposite coorientations, then $\sigma_{\theta_2}(v) = -\sigma_{\theta_1}(v)$ for each vertex and so
\[\Rot_{\theta_2} = -\frac{1}{2}\left(\sigma_{\theta_1}(v_1) - \sigma_{\theta_1}(v_2) \right) = \mp 1\]
which is a contradiction.  Thus $R_{\theta_1}$ and $R_{\theta_2}$ must have the same coorientation and therefore Theorem \ref{thrm:complete} implies that $\theta_1,\theta_2$ are isotopic.
\end{proof}

\begin{proposition}
\label{prop:GL-stab}
Every topologically trivial Legendrian embedding $\theta: \Theta \rightarrow (S^3,\xi_{std})$ is a stabilization of some $G_l$.
\end{proposition}

\begin{proof}
Let $\theta$ be such a graph.  Note that Theorem \ref{thrm:complete} and Lemma \ref{lemma:tb-determines-contact} together imply that there are at most 2 Legendrian graphs with the same $\tbbf,\rotbf$.  If $\Rot_{\theta} = 0$, then $\theta$ must be one of the two constructed in Lemma \ref{lemma:realization-rot0}.  If $\Rot_{\theta} = \pm 1$, then Lemma \ref{lemma:realization-rot1} implies $\theta$ must be Legendrian isotopic to the one constructed in Lemma \ref{lemma:realization}.
\end{proof}

\begin{corollary}
The set of nondestabilizeable, topologically trivial Legendrian $\Theta$-graphs is precisely $\{G_l\}_{l \geq -\frac{1}{2}}$.
\end{corollary}

Theorem \ref{thrm:complete} combined with Lemmas \ref{lemma:realization-rot0} and \ref{lemma:realization-rot1} classify topologically trivial Legendrian embeddings of $\Theta$.  We now address the classification of the images of these embeddings.  

Relabeling the edges and vertices of the image $\theta(\Theta)$ corresponds to replacing $\theta$ with $\theta \circ \phi$ for some $\phi \in \text{Aut}(\Theta)$.  The automorphisms of $\Theta$ are permutations of the vertices and edges and so the automorphism group is $S_2 \times S_3$.  Under $\phi$, the image or orientation of a cycle may change and so the invariants $\tbbf_{\theta},\rotbf_{\theta}$ change as well after relabeling.

Let $\phi_v$ denote the automorphism that fixes the edges and swaps the vertices and let $\phi_i$ denote the automorphism that fixes the vertices and transposes the edges $e_i,e_{i+1}$.  The four automorphisms $\{\phi_v,\phi_1,\phi_2,\phi_3\}$ generate $\text{Aut}(\Theta)$.  The automorphism $\phi_v$ acts by reversing the orientation on cycles
\[\phi_v(\gamma_j) = - \gamma_j \text{ for }j=1,2,3\]
and for $i=1,2,3$, the automorphism $\phi_i$ acts on the set of oriented cycles by
\begin{align*}
\phi_i(\gamma_i) &= - \gamma_i & \phi_i(\gamma_{i \pm 1}) & = - \gamma_{i \mp 1}
\end{align*}

To describe this effect on $\tbbf,\rotbf$, we describe the induced representation $\rho$ of $\text{Aut}(\Theta)$ on $\ZZ^6$.  This representation is defined so that for every $\phi \in \text{Aut}(\Theta)$, the classical invariants satisfy
\[(\tbbf_{\theta \circ \phi},\rotbf_{\theta \circ \phi}) = \rho(\phi) \cdot (\tbbf_{\theta},\rotbf_{\theta})\]

Let $I_k$ denote the trivial representation of $S_k$ on $\ZZ^3$, i.e. $\rho(\tau)$ is the identity for all $\tau$.  Let $\text{Sgn}_k$ denote the sign representation of $S_k$ on $\ZZ^3$, i.e. $\rho(\tau)$ is multiplication by the sign of the permutation $\tau$.  Let $p_{\zeta}$ denote the shifted permutation representation of $S_3$ on $\ZZ^3$, i.e. if $\tau_{i}$ is an elementary transposition in $S_3$ swapping $i$ and $i+1$, then $p_{\zeta}(\tau_{i})$ transposes the coordinates $z_{i+1},z_{i+2}$.  Define representations $\rho_2,\rho_3$ of $S_2,S_3$, respectively, on $\ZZ^6$ by
\begin{align*}
\rho_2 &:= I_2 \oplus \text{Sgn}_2 \\
\rho_3 &:= (I_3 \oplus \text{Sgn}_3) \cdot (p_{\zeta} \oplus p_{\zeta})
\end{align*}
and extend to a representation $\rho = \rho_2 \times \rho_3$ of $S_2 \times S_3$.  Identifying $\text{Aut}(\Theta)$ with $S_2 \times S_3$ and viewing $(\tbbf,\rotbf) = (\tb_1,\tb_2,\tb_3,\rot_1,\rot_2,\rot_3)$ as a vector in $\ZZ^6$, we see that
\begin{align*}
\rho(\phi_v) \cdot (\tb_1,\tb_2,\tb_3,\rot_1,\rot_2,\rot_3) &= (\tb_1,\tb_2,\tb_3,-\rot_1,-\rot_2,-\rot_3) \\
\rho(\phi_1) \cdot (\tb_1,\tb_2,\tb_3,\rot_1,\rot_2,\rot_3) &= (\tb_1,\tb_3,\tb_2,-\rot_1,-\rot_3,-\rot_2) \\
\rho(\phi_2) \cdot (\tb_1,\tb_2,\tb_3,\rot_1,\rot_2,\rot_3) &= (\tb_3,\tb_2,\tb_1,-\rot_3,-\rot_2,-\rot_1) \\
\rho(\phi_3) \cdot (\tb_1,\tb_2,\tb_3,\rot_1,\rot_2,\rot_3) &= (\tb_2,\tb_1,\tb_3,-\rot_2,-\rot_1,-\rot_3)
\end{align*}
From the above discussion, it follows that this is the correct effect of the classical invariants.

\begin{proposition}
Suppose that $\theta,\theta'$ are topologically trivial Legendrian $\Theta$-graphs such that $$\rho(\phi) \cdot (\tbbf_{\theta},\rotbf_{\theta}) = (\tbbf_{\theta'},\rotbf_{\theta'})$$ for some $ \phi \in \text{Aut}(\Theta)$.  Then
\begin{enumerate}
\item If $\Rot_{\theta}, \Rot_{\theta'} \neq 0$ then $\theta,\theta'$ are Legendrian isotopic up to relabeling.
\item If $\Rot_{\theta},\Rot_{\theta'} = 0$, then $\theta,\theta'$ are Legendrian isotopic up to relabeling if and only if either
\begin{enumerate}
\item $\sigma_{\theta \circ \phi}(v_1) = \sigma_{\theta'}(v_1)$, or
\item there is a transposition $\tau \in S_3$ such that $\rho_3(\tau)$ fixes $(\tbbf_{\theta'},\rotbf_{\theta'})$.
\end{enumerate}
\end{enumerate}
\end{proposition}

\begin{proof}
First, suppose that $\Rot_{\theta}, \Rot_{\theta'} \neq 0$.  Then by assumption $(\tbbf_{\theta \circ \phi},\rotbf_{\theta \circ \phi}) = (\tbbf_{\theta'},\rotbf_{\theta'})$ and so by Lemma \ref{lemma:realization-rot1}, the graphs $g \circ \phi$ and $h$ are Legendrian isotopic.

Second, suppose that $\Rot_{\theta} = \Rot_{\theta'} = 0$.  If $\sigma_{\theta \circ \phi}(v_1) = \sigma_{\theta'}(v_1)$, then the Legendrian ribbons $R_{\theta \circ \phi}$ and $R_{\theta'}$ have the same coorientation, and by Theorem \ref{thrm:complete}, $g \circ \phi$ and $h$ are Legendrian isotopic.  

Instead, suppose that $\sigma_{\theta \circ \phi}(v_1) \neq \sigma_{\theta'}(v_1)$. If the required transposition $\tau$ exists, then $(\tbbf_{\theta \circ \phi},\rotbf_{\theta \circ \phi}) = (\tbbf_{\theta' \circ \tau},\rotbf_{\theta' \circ \tau})$ and since $\tau$ is an odd permutation, $\sigma_{\theta \circ \phi}(v_1) = \sigma_{\theta' \circ \tau}(v_1)$ and so $\theta \circ \phi$ and $\theta' \circ \tau$ are Legendrian isotopic.  Conversely, $\theta$ and $h$ are isotopic up to relabeling if and only if there exists some automorphism $\psi$ such that $\theta \circ \phi \circ \psi$ and $\theta'$ are Legendrian isotopic.  In particular, this means that $\sigma_{\theta \circ \phi \circ \psi}(v_1) = \sigma_{\theta'}(v_1)$.  Consequently, $\psi$ must be an odd permutation in $S_3$ and the odd permutations of $S_3$ are exactly the transpositions.
\end{proof}

\section{Graph moves and stabiliziations} 
\label{sec:moves}
In this section we introduce two moves between different Legendrian graphs, and show how the $G_l$ graphs are related by these moves.  

Fuchs and Tabachnikov \cite{Fuchs-Tabachnikov} showed that if $L_1,L_2$ are Legendrian knots that are topologically isotopic, then they are Legendrian isotopic after applying a sequence of stabilizations.  
This is not the case for Legendrian graphs, stabilizations along edges of Legendrian graphs are not a sufficient collection of moves to accomplish this.  
Recall, for Legendrain graphs stabilization is defined in a similar way and can happen on any edge.  
A simple example of Legendrian graphs that are not related by are the pair Legendrian $\Theta$-graphs $G_{-\frac{1}{2}}$ and $G_0$.  
These two graphs have a different cyclic ordering of the edges around one of the vertices.  
There is no way to changed the cyclic ordering at a vertex with only edge stabilizations.  
Here we will define two new moves: vertex stabilization and vertex twist.  

%
%


\begin{figure}[htpb!]
\begin{center}
\begin{picture}(300, 120)
\put(0,0){\includegraphics[width=4.3in]{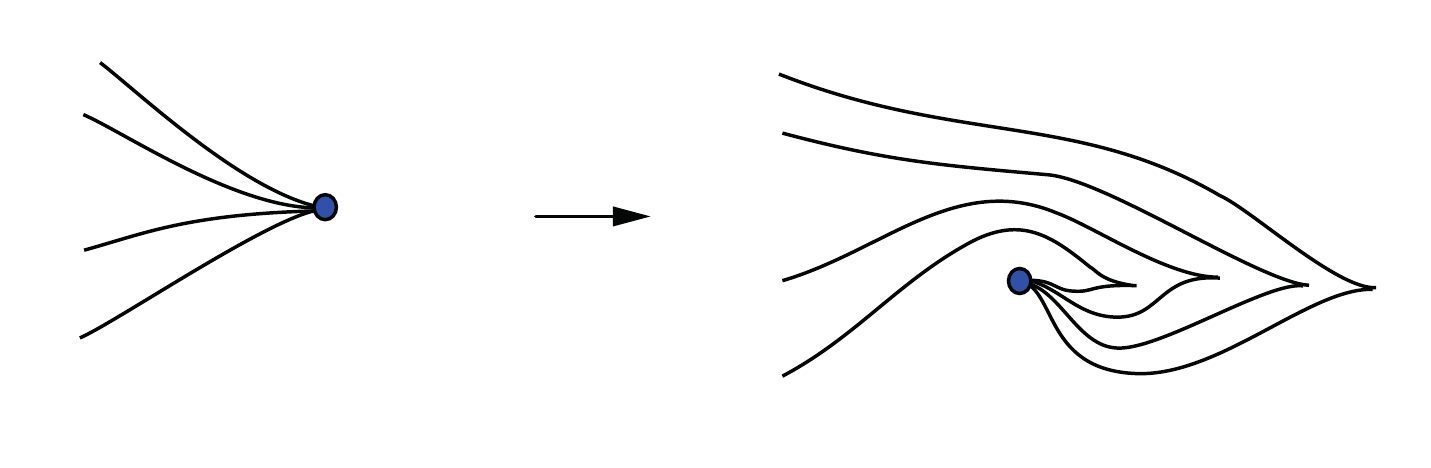}}
\put(115,60){\small \textsc{Vstab}}
\end{picture}
\caption{A vertex stabilization in the front projection on a valence 4 vertex.  }\label{fig:vertex-stabilizations}
\end{center}
\end{figure}

\begin{definition}
A {\it vertex stabilization} is defined for vertices of valence 3 or more, by front projection diagrams shown in  Figure~\ref{fig:vertex-stabilizations}; here the neighborhood of a vertex is replaced with a vertex and arcs where a half stabilization is introduced to each edge.  
\end{definition}

Figure~\ref{fig:vertex-stabilizations} shows a stabilization of a valance 4 vertex.  
Vertex stabilization reverses the cyclic order of edges around the vertex in the contact plane.  
Note that in a neighborhood of a vertex, all edges can be moved to the left using Reidemeister V moves.

\begin{remark}
The choice to have the edges kind downwards rather than upwards is inconsequential.  In Figure \ref{fig:VstabBraid} we show a sequence of Legendrian graph Reidemeister moves to go from the standard vertex stablization form to a braided representation.  
From the braided from one can move in a similar way to the projection with all edges kinked the other way.  
\end{remark}


\begin{figure}[htpb!]
\begin{center}
\begin{picture}(389, 150)
\put(0,0){\includegraphics[width=5.5in]{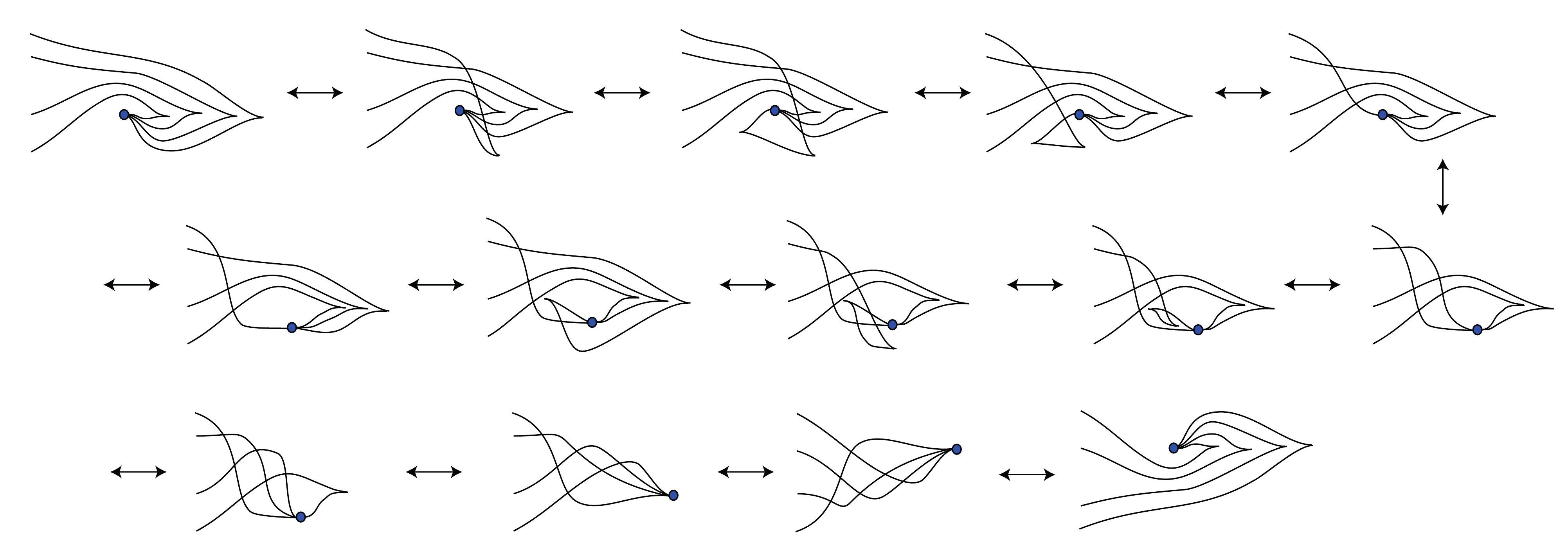}}
\put(70,122){3x II}
\put(154,122){V}
\put(233,122){III$_v$}
\put(313,122){I}

\put(367,92){A}

\put(20,73){p. iso}
\put(107,73){V}
\put(184,73){III$_v$}
\put(258,73){II}
\put(331,73){I}

\put(32,25){A}
\put(107,25){V}
\put(176,25){4x III}
\put(258,25){$\overline{\textrm{B}}$}
\end{picture}
\caption{ Moving between the standard projection of a vertex stabilization, the braided from, and the mirrored projection.    
The double sided arrows indicate Legendrian graph Reidemeister moves, sets of Legendrian graph Reidemeister moves, or planar isotopy.  
The set of moves in the second row (together called A), shows how to move a edge $e_i$ ($i\neq 1,n$) into the braid formation.  
The set of moves to get from the standard projection of a vertex stabilization, the braided from (shown third from the end) we will call B.  
The symbol $\overline{\textrm{B}}$ is the similar set of Legendrian graph Reidemeister moves going from our mirrored braid to the mirrored projection.  }
\label{fig:VstabBraid}
\end{center}
\end{figure}
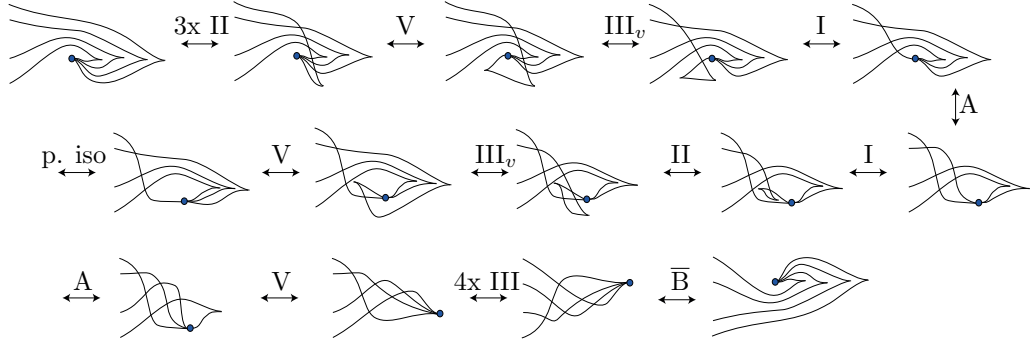

\begin{observation}
Vertex stabilization is not intrinsic to the vertex itself.  The front diagram picks out the top and bottom edges as a distinguished pair of edges.  The Legendrian isotopy class of the resulting graph depends on this distinguished pair.  By Reidemeister moves, we can cyclically rotate the edges as in Figure \ref{fig:Vstab2} and then apply a vertex stabilization.  However, this is not Legendrian isotopic to the original vertex stabilization.

To see this, let the edges incident to $v$ be labeled $e_i$, starting with the top edge and going to the bottom (clockwise in the cyclic ordering).  
Let $\alpha_i$ be the arc $e_i$ followed by $e_{i+1}$, and $\alpha_n$ be the arc $e_n$ followed by $e_1$.  
This gives an orientation on the arcs, where each edge appears once with each orientation.  
After the vertex stablilzation the arcs $\alpha_i$ with $i=1,\dots, n-1$ will be negatively stabilized and the arc $\alpha_n$ will be positively stabilized.  See Figure \ref{fig:VstabUpDown}.  In particular, any cycle containing the distinguished pair of edges will be stabilized with the opposite sign.

Thus, for a valence $n$ vertex, there are $n$ different possible vertex stabilizations.  See Figure \ref{fig:Vstabs}
\end{observation}


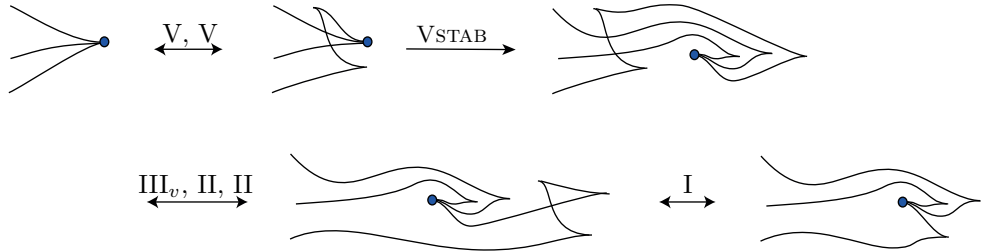
\begin{figure}[htpb!]
\begin{center}
\begin{picture}(400, 130)
\put(0,0){\includegraphics[width=5.5in]{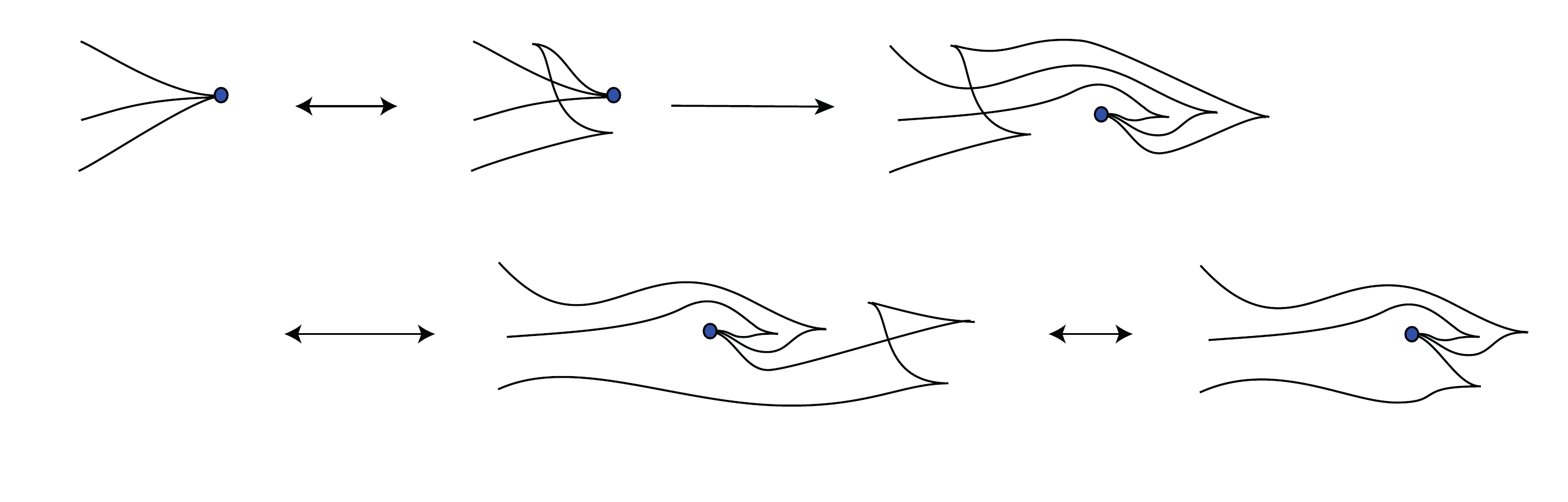}}

\put(78, 103){V, V}
\put(174, 103){\small\textsc{Vstab}}
\put(69, 46){III$_v$, II, II}
\put(275, 46){I}
\end{picture}
\caption{Obtaining the second type of vertex stabilization $S_2(v)$ using only the vertex stabilization move and Reidemeister moves.}
\label{fig:Vstab2}
\end{center}
\end{figure}


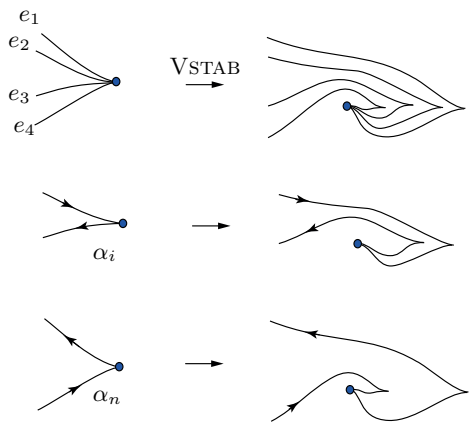
\begin{figure}[htpb!]
\begin{center}
\begin{picture}(150, 170)
\put(0,0){\includegraphics[width=2.5in]{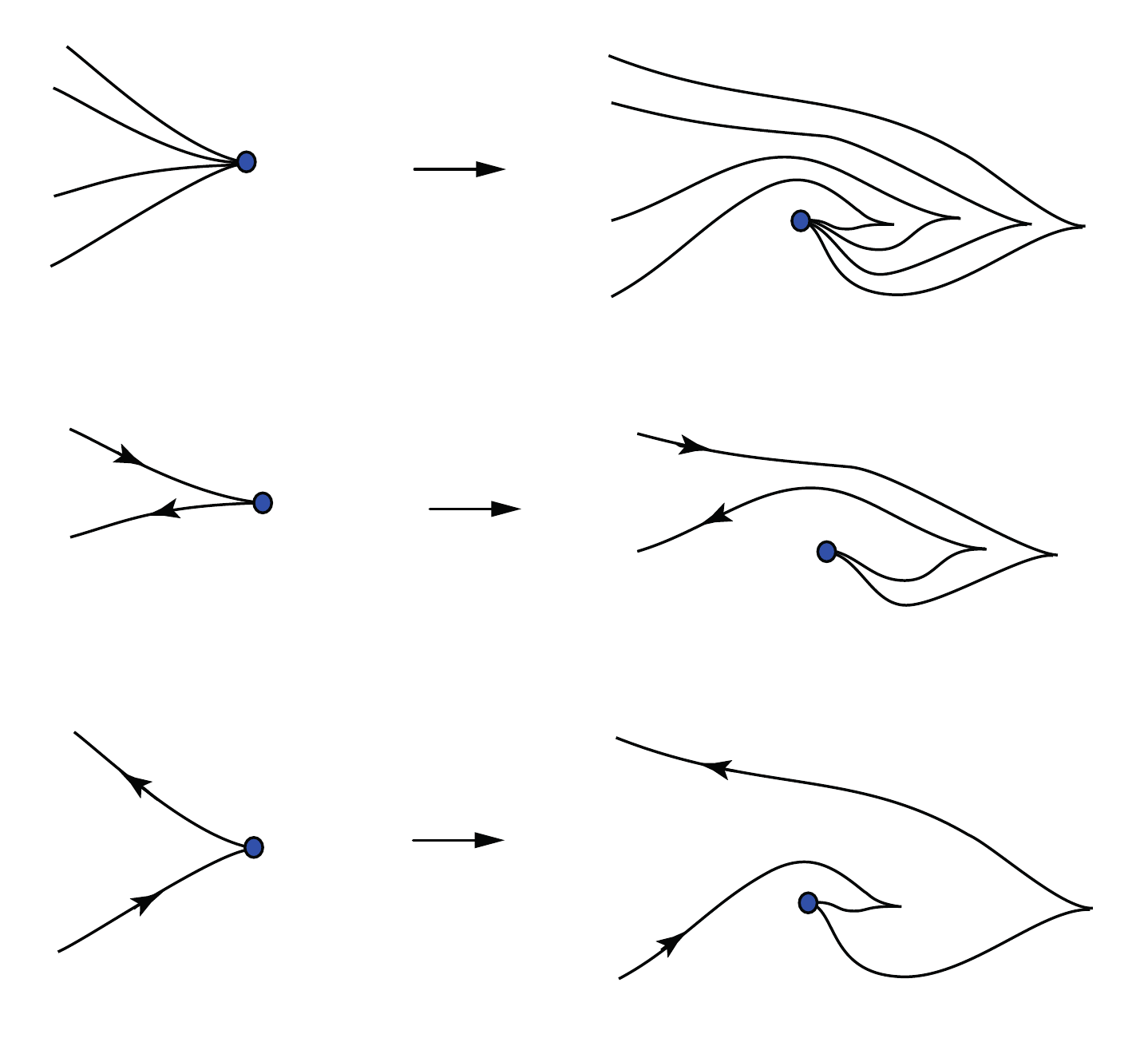}}
\put(59,142){\small\textsc{Vstab}}

\put(2, 162){\small $e_1$}
\put(-2, 152){\small $e_2$}
\put(-2, 132){\small $e_3$}
\put(0, 120){\small $e_4$}

\put(30, 72){\small $\alpha_i$}
\put(30, 18){\small $\alpha_n$}
\end{picture}
\caption{ The top row shows vertex stabilization on a 4 valent vertex, the lower two show how the arcs defined by pairs of edges are changed under vertex stabilization.   
With the edges labeled as indicated, let $\alpha_i$ be the arc $e_i$ followed by $e_{i+1}$ (for $i=1,2,3$), and $\alpha_n$ be the arc $e_4$ followed by $e_1$.}
\label{fig:VstabUpDown}
\end{center}
\end{figure}


\begin{figure}[htpb!]
\begin{center}
\begin{picture}(400, 60)
\put(0,0){\includegraphics[width=5.5in]{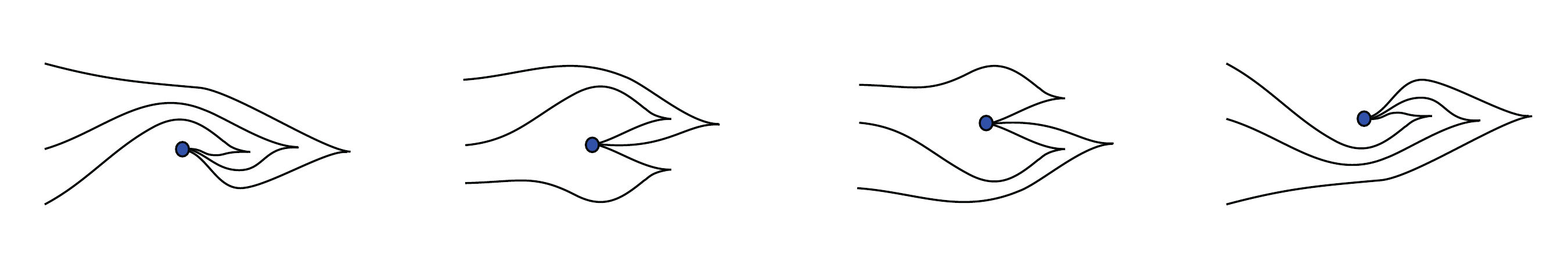}}

\put(35, 1){\small $S_1(v)$}
\put(135, 1){\small $S_2(v)$}
\put(235, 1){\small $S_3(v)$}
\put(335, 1){\small $S_1(v)$}
\end{picture}
\caption{The three different stabilizations for a trivalent vertex.  The equivalence of the first and last projections is shown in Figure \ref{fig:VstabBraid}.}
\label{fig:Vstabs}
\end{center}
\end{figure}
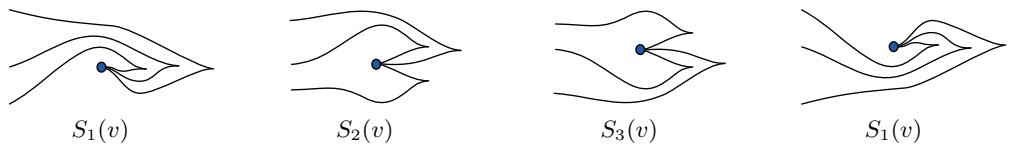

The different vertex stabilizations are obtained by having a different pair of edges in the top and bottom positions.  
Any pair of neighboring edges can be moved to these positions using Reidemeister V moves.  
Figure \ref{fig:Vstabs} shows the three different vertex stabilizations for a trivalent vertex.  
Figure \ref{fig:Vstab2} shows how to obtain the second stabilization $S_2(v)$ through Reidemeister moves and the vertex stabilization.  
Thus the single move shown in Figure~\ref{fig:vertex-stabilizations} for valence $n$ vertices ($n\geq3$), together with the Legendrian graph Reidemeister moves gives all possible vertex stabilizations.  

\begin{definition}
A {\it vertex twist} is defined for vertices of valence 3 or more, by front projection diagrams shown in Figure~\ref{fig:Vtwist}; here the neighborhood of a vertex is replaced with a vertex and arcs where two neighboring edges have switched cyclic ordering around the vertex and one edges crosses over the other as shown.  
\end{definition}

\begin{remark}
If we restrict  to trivalent graphs the vertex twist is redundant.  
Both positive and negative vertex twists can be obtained by a vertex stabilization, an edge destabilization, and Reidemeister moves.  
See Figure \ref{fig:3Vtwist}.
\end{remark}


\begin{figure}[htpb!]
\begin{center}
\begin{picture}(150, 100)
\put(0,0){\includegraphics[width=2.5in]{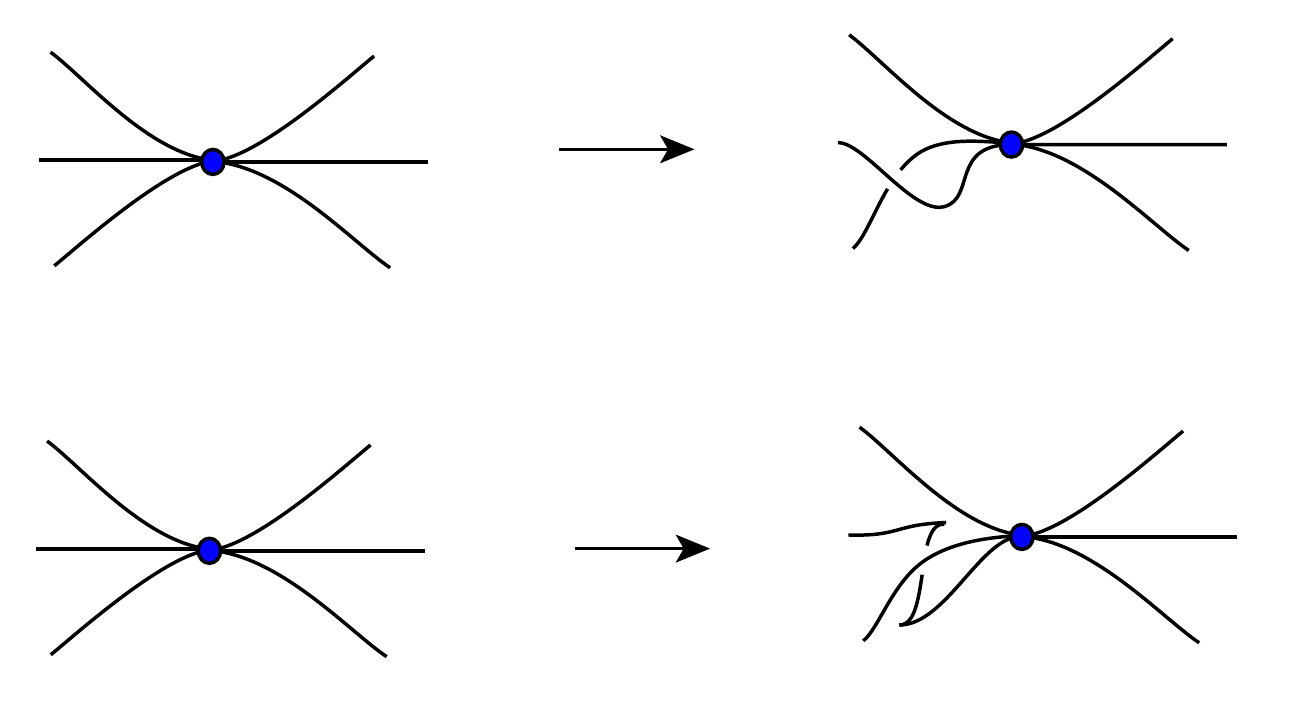}}
\put(75,83){\small\textsc{twist}$_+$}
\put(75,25){\small\textsc{twist}$_-$}
\end{picture}
\caption{Positive and negative vertex twist.  }
\label{fig:Vtwist}
\end{center}
\end{figure}

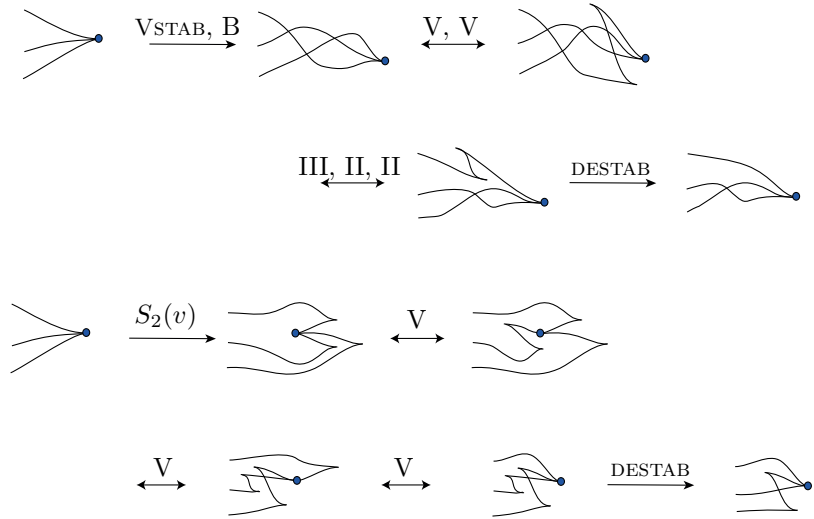
\begin{figure}[htpb!]
\begin{center}
\begin{picture}(300, 220)
\put(0,-10){\includegraphics[width=4.5in]{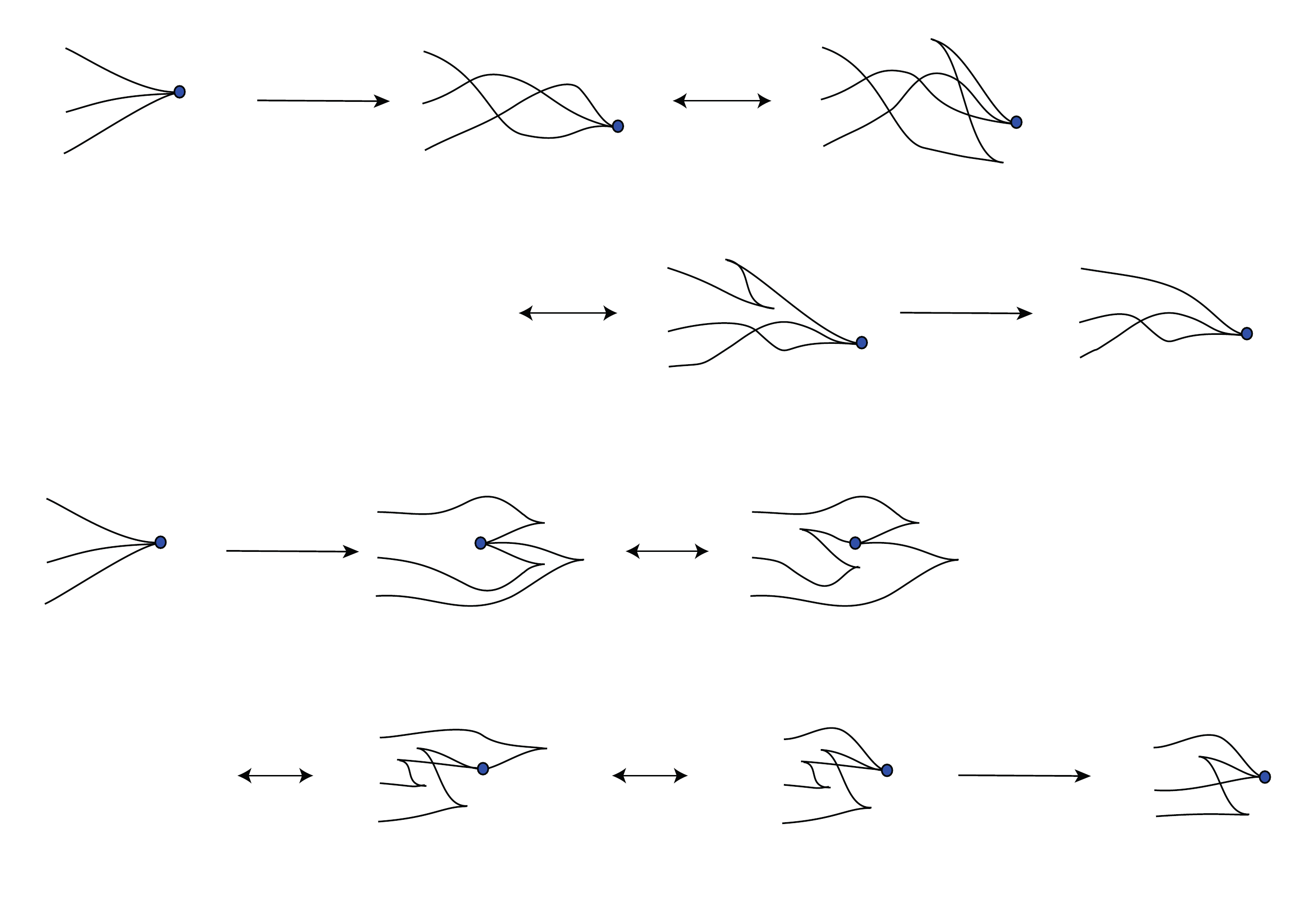}}
\put(58, 191){{\small\textsc{Vstab}}, B}
\put(167, 191){V, V}

\put(120, 138){III, II, II}
\put(223, 138){\small\textsc{destab}}

\put(58,82){$S_2(v)$}
\put(161, 80){V}

\put(65,24){V}
\put(156,24){V}
\put(238,24){\small\textsc{destab}}
\end{picture}
\caption{Obtaining the vertex twist moves for trivalent vertices, through vertex stabilization, an edge destabilization, and Reidemeister moves.}
\label{fig:3Vtwist}
\end{center}
\end{figure}

Now that we have defined all of the moves on Legendrian graphs, we will look at how the $G_i$'s are related by such moves.  

\begin{lemma}
The graph $G_{l + \frac{1}{2}}$ can be obtained from $G_{l}$ by a vertex stabilization and then an edge destabilization.
\end{lemma}\label{lemma:G_ltoG_l+}

\begin{proof}
Consider an arbitrary $G_l$, as shown in Figure \ref{fig:G_ltoG_l+}.  
A vertex stabilization is done on the right-most vertex, and is moved to braid form.  
Then Reidemeister moves are done as indicated to move to a $G_{l + \frac{1}{2}}$ with the distinguished edge stabilized.  
This completes the lemma.  
\end{proof}

\begin{figure}[htpb!]
\centering
\labellist
	\small\hair 3pt
	\pinlabel \textsc{Vstab} at 205 357
	\pinlabel B at 530 350
	
	\pinlabel III at 37 220
	\pinlabel V at 332 220
	\pinlabel III$_V$ at 613 220

	\pinlabel I at 71 85
	\pinlabel II at 370 85
	\pinlabel V at 657 85
	
	\small\hair 2pt
	\pinlabel $l+\frac{1}{2}$ at 93 327
	\pinlabel $l+\frac{1}{2}$ at 325 327
	\pinlabel $l+\frac{1}{2}$ at 660 327

	\pinlabel $l+\frac{1}{2}$ at 169 195
	\pinlabel $l+1$ at 463 200
	\pinlabel $l+1$ at 737 195

	\pinlabel $l+1$ at 207 59
	\pinlabel $l+1$ at 492 54
	\pinlabel $l+1$ at 791 54

\endlabellist
\includegraphics[width=1\textwidth]{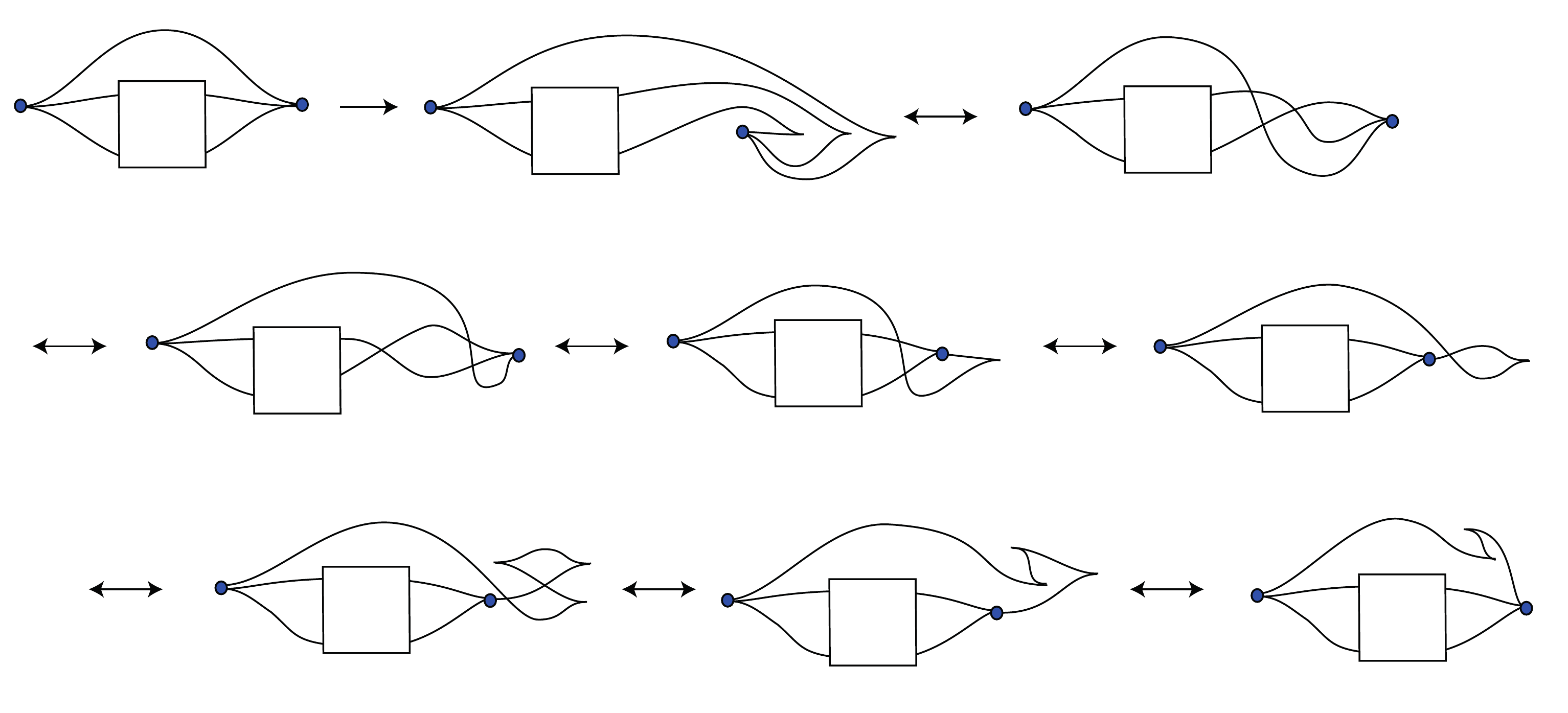}
\caption{The result of a single vertex stabilization on $G_l$.  The first arrow shows a vertex stabilization.  
The double sided arrows indicate Legendrian graph Reidemeister moves or sets of Legendrian graph Reidemeister moves.  
The symbol {\bf B} is for the set of Legendrian graph Reidemister moves need to move from the standard front projection of a vertex stabilization to the braided form.  See Figure \ref{fig:VstabBraid}.   }
\label{fig:G_ltoG_l+}
\end{figure}
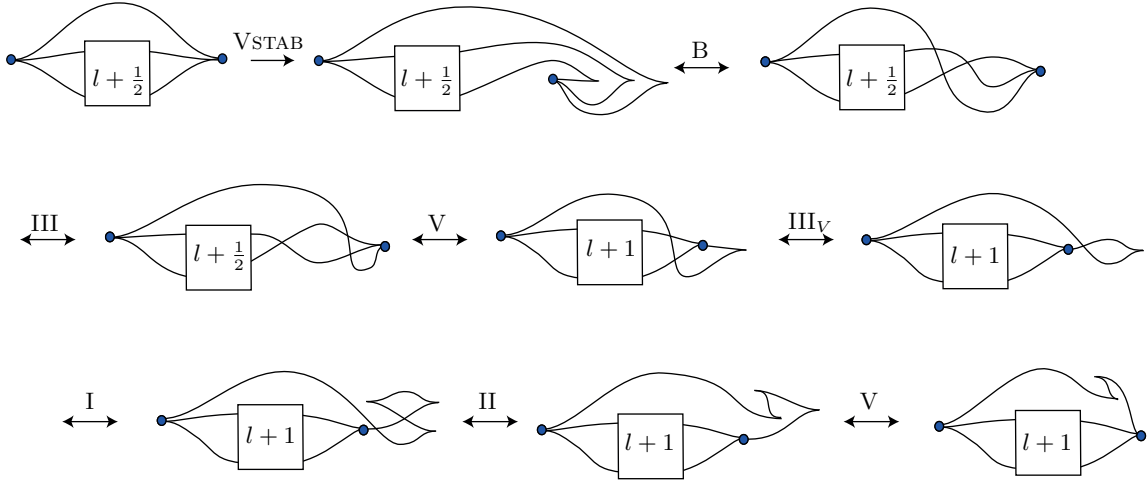

\begin{lemma}
The graph $G_{l'}$ can be obtained from $G_{l}$ by a sequence of an edge stabilization (or destabilization) if and only if $l - l' \in \ZZ$.
\end{lemma}

\begin{proof}
Suppose that $l-l'$ were not an integer, then Lemma \ref{lemma:G_ltoG_l+} they are related by an odd number of vertex stabilizations and some edge destabilizations.  
Since they are related by an odd number of vertex destabilizations, the two graphs have a different cyclic ordering of the edges around one of the vertices.  
Thus they cannot be related by edge stabilization / destabilization.  

Now suppose $l - l' \in \ZZ$.  
Without loss of generality suppose $l<l'$.  
Consider an arbitrary $G_l$, as shown in Figure \ref{fig:G_ltoG_l'}.  
Two edge stabilizations are done on the lower two edges.  
Then Reidemeister moves are done as indicated to move to a $G_{l + 1}$ with the distinguished edge stabilized.  
This edge can be destabilized to obtain $G_{l+1}$.  
This process is repeated to obtain $G_{l'}$ from $G_l$.  
\end{proof}

\begin{figure}[htpb!]
\centering
\labellist
	\small\hair 3pt
	\pinlabel {2x \textsc{stab}} at 208 424
	\pinlabel {II, II} at 463 420
	\pinlabel V at 708 419

	\pinlabel V at 37 318
	\pinlabel {II, II} at 295 316
	\pinlabel {II, II, V, V} at 580 316
	
	\pinlabel III at 40 204
	\pinlabel V at 335 204
	\pinlabel III$_v$ at 615 204

	\pinlabel I at 71 82
	\pinlabel II at 373 82
	\pinlabel V at 657 82
	
	\small\hair 2pt
	\pinlabel $l+\frac{1}{2}$ at 93 403
	\pinlabel $l+\frac{1}{2}$ at 325 401
	\pinlabel $l+\frac{1}{2}$ at 583 393
	\pinlabel $l+\frac{1}{2}$ at 825 393

	\pinlabel $l+\frac{1}{2}$ at 157 293
	\pinlabel $l+\frac{1}{2}$ at 422 293
	\pinlabel $l+\frac{1}{2}$ at 714 291

	\pinlabel $l+\frac{1}{2}$ at 157 178
	\pinlabel $l+\frac{3}{2}$ at 463 182
	\pinlabel $l+\frac{3}{2}$ at 737 180

	\pinlabel $l+\frac{3}{2}$ at 208 56
	\pinlabel $l+\frac{3}{2}$ at 493 49
	\pinlabel $l+\frac{3}{2}$ at 792 51

\endlabellist
\includegraphics[width=1\textwidth]{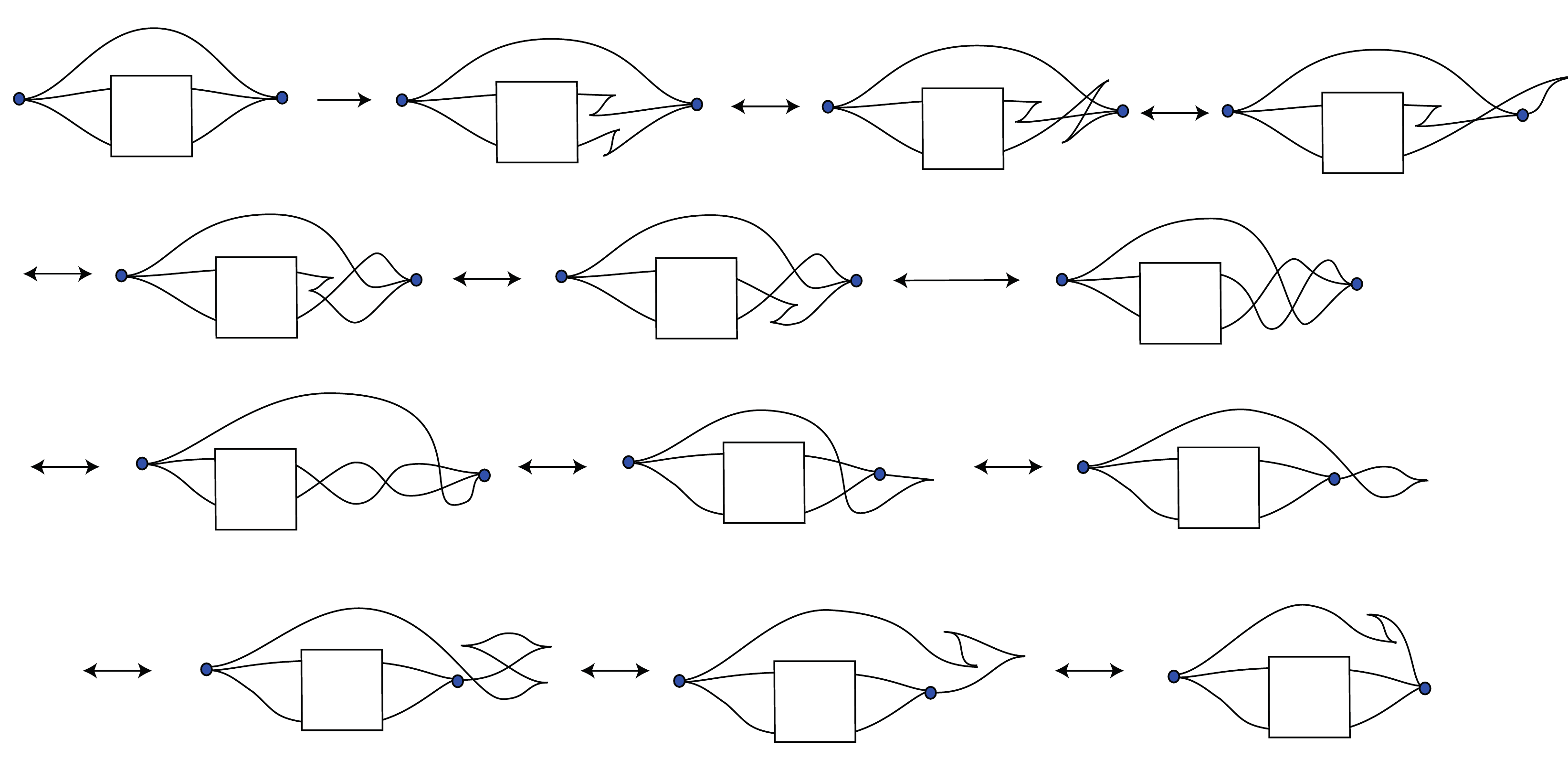}
\caption{ Moving between $G_{l}$ and $G_{l+1}$ with a stabilized edge.  
The first arrow shows two edge stabilizations.  
The double sided arrows indicate Legendrian graph Reidemeister moves or sets of Legendrian graph Reidemeister moves as indicated.   }
\label{fig:G_ltoG_l'}
\end{figure}
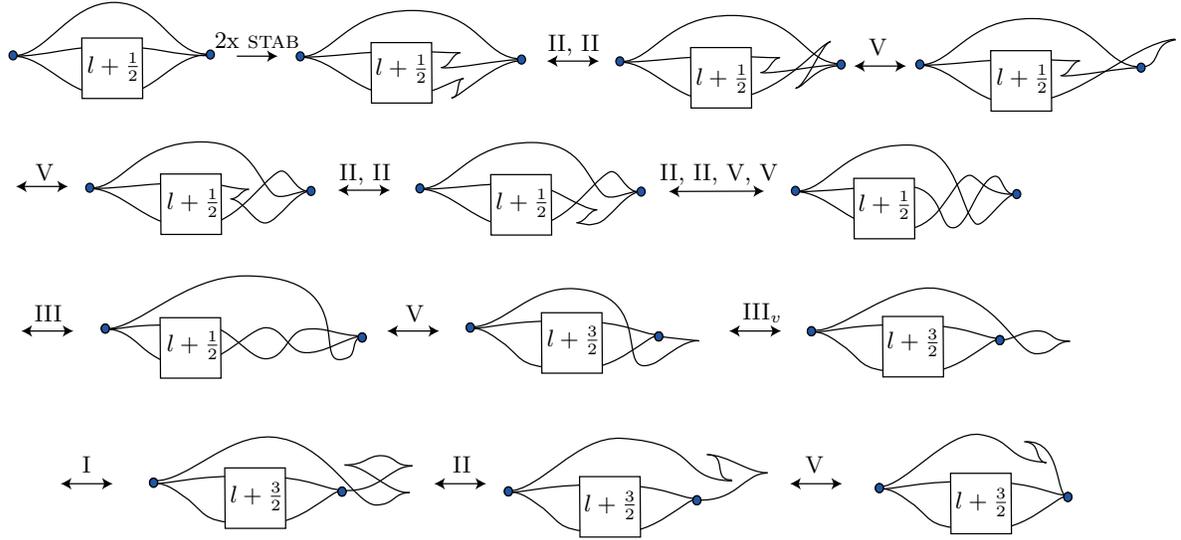

\section{Nondestabilizable graphs} 
\label{sec:minors}

In this subsection, we prove that a planar graph has infinitely many nondestabilizeable realizations if it contains a subdivision of $\Theta$ or $S^1\vee S^1$ as a subgraph.  The proof implicitly uses convex surface theory, although we will not use the full theory here.  The only fact we require is the following instance of the Legendrian Realization Principle \cite{Honda1}.  

\begin{proposition}[Legendrian Realization Principle]
\label{prop:LERP}
Let $S$ be the unit sphere in $(\RR^3,\xi_{std})$, let $\Gamma = S \cap \{z = 0\}$ be the equator of $S$, and let $G$ be a graph embedded on $S$.  Then there exists a $C^{\infty}$ small perturbation $\phi$ of $S$, fixing $\Gamma$, such that $\phi(G)$ is Legendrian.  In addition, if $\gamma$ is a cycle of $G$ then $\tb(\gamma) = -\frac{1}{2} \# (\gamma \cap \Gamma)$.
\end{proposition}

Let $g: G \rightarrow (M,\xi)$ be a Legendrian embedding.  The graph $g$ has {\it Property N} if each edge is either (1) a cut edge, or (2) contained in a nondestabilizeable cycle.

\begin{proposition}
\label{prop:nondestab-criterion}
Let $g: G \rightarrow (M,\xi)$ be a Legendrian graph.  If $g$ has Property N, then $g$ is nondestabilzeable.
\end{proposition}

\begin{proof}
If $g$ admits a destabilization along $e$, then every cycle containing $e$ admits a destabilization.
\end{proof}

A edge $e$ (respectively vertex $v$) of a connected graph $G$ is a {\it cut edge} (respectively {\it cut vertex}) if $G \smallsetminus e$ ($ G \smallsetminus v$) is disconnected.

\begin{proposition}
\label{prop:realization-propN}
Let $G$ be an abstract planar graph.  Then there exists a topologically trivial Legendrian embedding $g: G \rightarrow (S^3,\xi_{std})$ that has Property N.
\end{proposition}

\begin{proof}
Fix an embedding of $G$ in $S^2$ and let $\Gstar$ be the dual graph of $G$ in $S^2$.  Choose a spanning tree $\Tstar$ for $\Gstar$ and let $\Gamma$ be the boundary of a tubular neighborhood of $\Tstar$ in $S^2$.  Choose a map $ S^2 \rightarrow S$ that sends $\Gamma$ to the equator of $S$, the unit sphere in $(\RR^3, \xi_{std})$.  By the Legendrian Realization Principle (Proposition \ref{prop:LERP}), a perturbation gives a topologically trivial Legendrian embedding $g$ of $G$.

In order to show that $g$ has Property N, we need to show that every non-cut edge is contained in a nondestabilizeable cycle.  By Proposition \ref{prop:LERP} and the classification of Legendrian unknots, this requires that each non-cut edge is contained in a cycle $\gamma$ satisfying $ \# \gamma \cap \Gamma = -2 \tb_g(\gamma) = 2$.  By the construction of $\Gamma$, this implies that the among all the edges comprising the cycle $\gamma$, exactly 1 corresponding dual edge is contained in the spanning tree $\Tstar$.

Fix a root vertex $F^*_{root}$.  For each face $F^*$ define $l(F^*)$ to be the length of the unique path in $\Tstar$ from $F^*_{root}$ to $F$.  Let $m$ denote the maximal length of any such path.  For each face $F^*$, define $d(F^*) = m - l(F^*)$.  We will prove that every edge in the boundary of every face is either a cut edge or contained in a cycle with $\tb = -1$ by induction on $d$.  

Let $F$ be a face such that $d(F^*) = 0$.  Then $F^*$ is a leaf of $T^*$.  In addition, since $F^*$ is leaf of $T^*$, it cannot be a cut vertex of $\Gstar$.  Let $\overline{F}$ be the closure of the corresponding face $F$.  The domain $\overline{F}$ must be a topological disk, since if it has multiple boundary components then $F^*$ is a cut vertex of $\Gstar$.  Then for each edge $e \in \del F$, either $e$ is a cut edge or $e$ lies in $\del \overline{F}$.  Let $\gamma = \del \overline{F}$. The cycle $\gamma$ contains exactly 1 edge whose dual lies in $\Tstar$ since $F^*$ is a leaf of $\Tstar$.  Therefore all edges in $\del F$ satisfy the condition.

Now, let $F$ be a face such that $d(F^*) = i > 0$ and suppose that all edges in the boundary of a face $F$ with $d(F^*) \leq i-1$ are either cut edges or contained in a cycle with $\tb = -1$.  If $F^* = F^*_{root}$, then every edge in the boundary of $F$ is either a cut edge or incident to another face $F'$.  Since $d(F^*_{root}) > d((F')^*)$, by induction all non-cut edges in $\del F$ are contained in a cycle of $\tb = -1$.

Suppose instead that $F^* \neq F^*_{root}$.  Removing $F^*$ from $T^*$ splits $T^*$ into two components, one containing the root $F^*_{root}$ and one containing faces with $d((F')^*) < i$.  Let $D$ be the union of the latter set of faces with $F$.  Let $\overline{D}$ be the closure of the union of the faces of $D$.  This domain must be a topological disk since otherwise the complement of the union of the dual faces in $\Gstar$ would have multiple components.  Let $\gamma = \del \overline{D}$ and let $e$ be an edge in $\del F$.  Then either (1) $e \in \gamma$, (2) $e$ is a cut edge, or (3) $e$ connects $F$ to some other face $F_j$ in $D$.  In the first two cases, it's clear that $e$ lies in a cycle with $\tb = -1$ or is a cut edge.  Finally, if (3), then by induction since $d(F') < i$ then $e$ lies in some other cycle with $\tb = -1$.
\end{proof}

\begin{customthm}{1.2}
Let $G$ be an abstract planar graph that contains a subdivision of $\Theta$ or $S^1 \vee S^1$ as a subgraph.  Then there exists infinitely many, pairwise-distinct, topologically trivial Legendrian embeddings $g: G \rightarrow (S^3,\xi_{std})$.
\end{customthm}

\begin{proof}
First suppose that $G$ contains a subdivision of $\Theta$ as a subgraph.  Fix an embedding of $G$ in $S^2$.  We can choose two vertices $v_1,v_2$ and three vertex-independent paths $p_1,p_2,p_3$ from $v_1$ to $v_2$.  Since $G$ is planar, then after possibly replacing $v_1,v_2,p_1,p_2,p_3$, we can assume that $v_1,v_2,p_2,p_3$ all lie in the boundary of some fixed face $F$.   Let $e_i$ denote the edge of $p_i$ incident to $v_1$.  After possibly modifying the embedding of $G$ in $S^2$, we can assume that $e_2,e_3$ are adjacent in the cyclic ordering of edges incident to $v_1$.

Since $p_2,p_3$ are vertex-independent, the dual face $F^*$ is not a leaf of $\Gstar$.  In addition, we can choose a spanning tree $\Tstar$ of $\Gstar$ containing $e^*_2$ and $e^*_3$.  Now, using the argument of Proposition \ref{prop:realization-propN}, we can use the spanning tree $T^*$ to construct a topologically trivial Legendrian embedding $g$ of $G$ with Property N.  Moreover, since $e^*_2,e^*_3 \in T^*$, every cycle $\gamma$ containing both $e_2$ and $e_3$ must have $\tb_g(\gamma) \leq -2$.

Let $g_k$ denote the graph obtained by performing $k$ positive vertex twists of $e_2$ and $e_3$ at $v_1$.  This is allowed since $e_2,e_3$ are adjacent at $v_1$.  If $\gamma$ is a cycle that does not contain both $e_2$ and $e_3$, then $g_k(\gamma)$ is Legendrian isotopic to $g(\gamma)$.  Consequently, if $\gamma$ is a cycle with $\tb_g(\gamma) = -1$ then $\tb_{g_k}(\gamma) = -1$ as well.   This implies that each $g_k$ has Property N and is nondestabilizeable.  However, if $\gamma$ is a cycle containing both $e_2$ and $e_3$, then $\tb_g(\gamma) - \tb_{g_k}(\gamma) = k$.  Thus, the Legendrian graphs $\{g_k\}$ are pairwise distinct.

Second, suppose that $G$ contains a subdivision of $S^1 \vee S^1$ as a subgraph and does {\it not} contain a subdivision of $\Theta$ as a subgraph.  Thus, there is a vertex $v$ and two vertex-independent paths $p_1,p_2$ from $v$ to itself.  Since $G$ does not contain a subdivided $\Theta$, every edge incident to $v$ either lies in a unique path from $v$ to $v$ or is a cut edge.  After possibly replacing the vertex and paths $v,p_1,p_2$ and the embedding of $G$ in $S^2$, we can assume that there exists edges $e_1,e_2$ that are consecutive in the cyclic ordering at $v$ and such that $e_i$ lies in $p_i$.

By a similar argument as above, we can find a nondestabilizeable Legendrian embedding $g$ of $G$. Let $g_k$ be the Legendrian graph obtained by performing $2k$ positive vertex twists of $e_1,e_2$.  Again, each $g_k$ is nondestabilizeable.  To distinguish these embeddings, we show that they have distinct contact framings.  Let $\widetilde{g}_k$ denote the restriction of $g_k$ to the subgraph $S^1 \vee S^1$.  The Legendrian ribbon $R_{\widetilde{g}_k}$ has a transverse pushoff $T_k$ with 3 components $T^1_k,T^2_k,T^3_k$.  See Figure \ref{fig:wedgeTpushoff}.  The linking number of $T^2_k$ with $T^3_k$ is $k$.  This implies that the embeddings $\{\widetilde{g}_k\}$ and therefore the embeddings $\{g_k\}$ are pariwise distinct.
\end{proof}


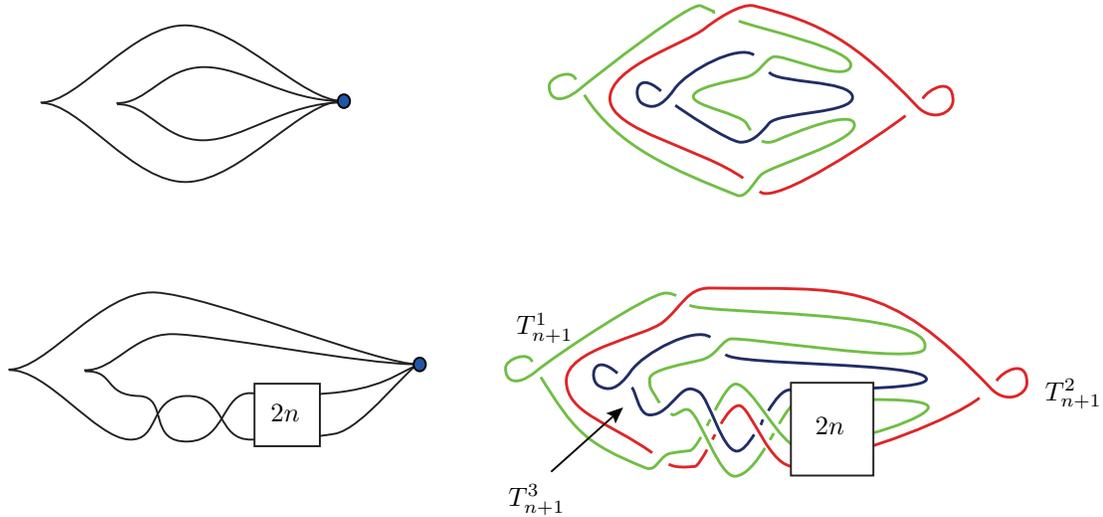
\begin{figure}[htpb!]
\begin{center}
\begin{picture}(400, 200)
\put(0,0){\includegraphics[width=5.5in]{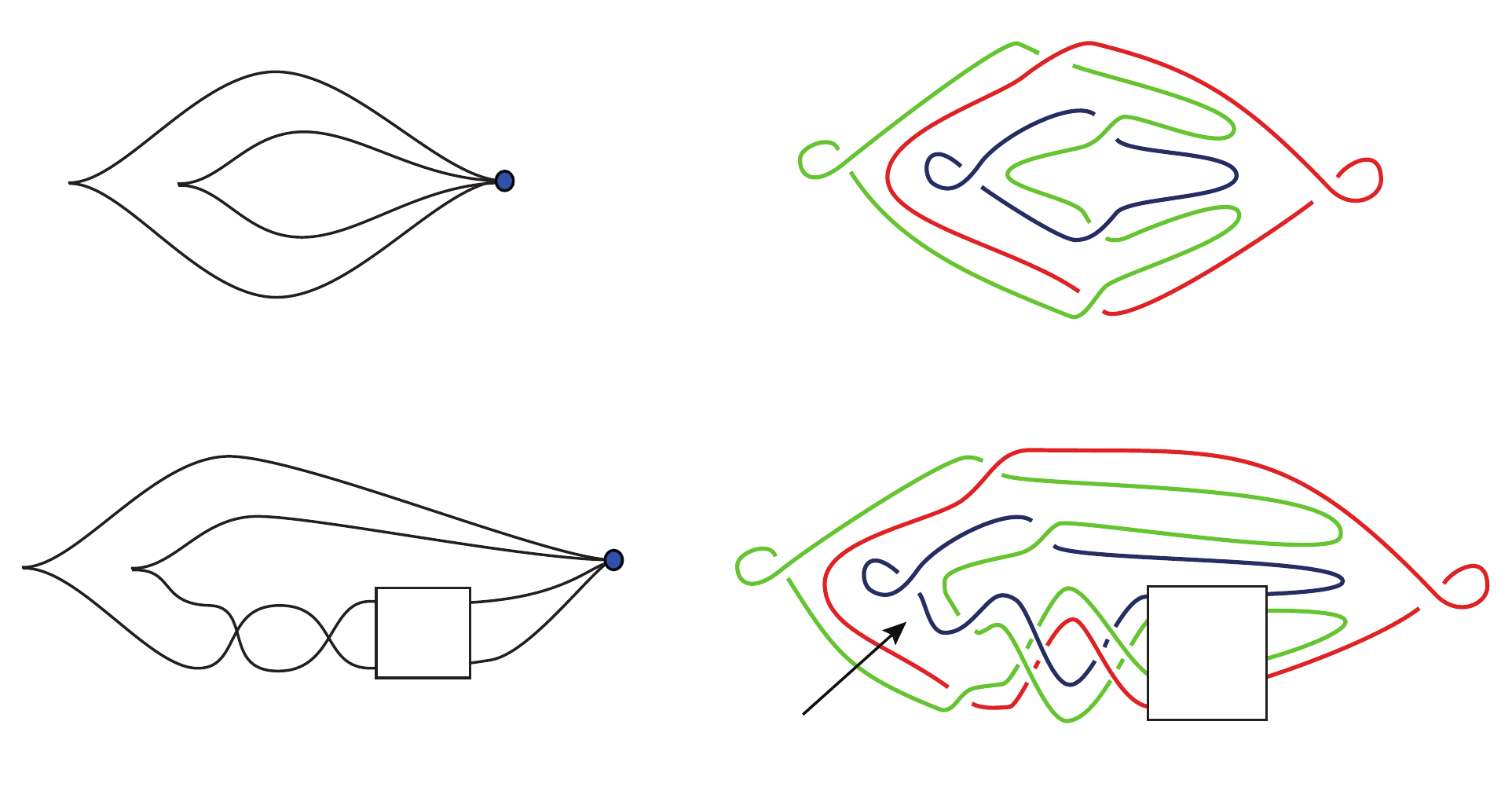}}

\put(105,42){$2n$}

\put(311,37){$2n$}
\put(198,75){$T^1_{n+1}$}
\put(398,50){$T^2_{n+1}$}
\put(195,10){$T^3_{n+1}$}

\end{picture}
\caption{On the left is a Legendrian embedding of $S^1 \vee S^1$ with both cycles maximal unknots (top) and the result of applying $2n+2$ vertex twists to the bottom two strands.  On the right are the transverse pushoffs of these graphs.  For all $n$, the transverse pushoff has 3 components and the linking number of $T^2_{n+1}$ and $T^3_{n+1}$ is $n+1$.}\label{fig:wedgeTpushoff}
\end{center}
\end{figure}

\bibliographystyle{alpha}
\nocite{*}
\bibliography{References}

\begin{thebibliography}{HKM07}

\bibitem[Ben83]{Bennequin}
Daniel Bennequin.
\newblock Entrelacements et \'equations de {P}faff.
\newblock In {\em Third {S}chnepfenried geometry conference, {V}ol. 1
  ({S}chnepfenried, 1982)}, volume 107 of {\em Ast\'erisque}, pages 87--161.
  Soc. Math. France, Paris, 1983.

\bibitem[BI09]{Baader}
Sebastian Baader and Masaharu Ishikawa.
\newblock Legendrian graphs and quasipositive diagrams.
\newblock {\em Ann. Fac. Sci. Toulouse Math. (6)}, 18(2):285--305, 2009.

\bibitem[BW00]{MR1847313}
Joan~S. Birman and Nancy~C. Wrinkle.
\newblock On transversally simple knots.
\newblock {\em J. Differential Geom.}, 55(2):325--354, 2000.

\bibitem[Che02]{Chekanov}
Yuri Chekanov.
\newblock Differential algebra of {L}egendrian links.
\newblock {\em Invent. Math.}, 150(3):441--483, 2002.

\bibitem[DG07]{Ding-Geiges}
Fan Ding and Hansj{\"o}rg Geiges.
\newblock Legendrian knots and links classified by classical invariants.
\newblock {\em Commun. Contemp. Math.}, 9(2):135--162, 2007.

\bibitem[EF09]{Eliashberg_Fraser}
Yakov Eliashberg and Maia Fraser.
\newblock Topologically trivial {L}egendrian knots.
\newblock {\em J. Symplectic Geom.}, 7(2):77--127, 2009.

\bibitem[EH01]{Etnyre_Honda}
John~B. Etnyre and Ko~Honda.
\newblock Knots and contact geometry. {I}. {T}orus knots and the figure eight
  knot.
\newblock {\em J. Symplectic Geom.}, 1(1):63--120, 2001.

\bibitem[Eli92]{Eliashberg92}
Yakov Eliashberg.
\newblock Contact {$3$}-manifolds twenty years since {J}. {M}artinet's work.
\newblock {\em Ann. Inst. Fourier (Grenoble)}, 42(1-2):165--192, 1992.

\bibitem[Eli93]{Eliashberg93}
Yakov Eliashberg.
\newblock Legendrian and transversal knots in tight contact {$3$}-manifolds.
\newblock In {\em Topological methods in modern mathematics ({S}tony {B}rook,
  {NY}, 1991)}, pages 171--193. Publish or Perish, Houston, TX, 1993.

\bibitem[ENV13]{MR3085098}
John~B. Etnyre, Lenhard~L. Ng, and Vera V{\'e}rtesi.
\newblock Legendrian and transverse twist knots.
\newblock {\em J. Eur. Math. Soc. (JEMS)}, 15(3):969--995, 2013.

\bibitem[Etn99]{Etnyre}
John~B. Etnyre.
\newblock Transversal torus knots.
\newblock {\em Geom. Topol.}, 3:253--268 (electronic), 1999.

\bibitem[FT97]{Fuchs-Tabachnikov}
Dmitry Fuchs and Serge Tabachnikov.
\newblock Invariants of {L}egendrian and transverse knots in the standard
  contact space.
\newblock {\em Topology}, 36(5):1025--1053, 1997.

\bibitem[Gir91]{Giroux91}
Emmanuel Giroux.
\newblock Convexit\'e en topologie de contact.
\newblock {\em Comment. Math. Helv.}, 66(4):637--677, 1991.

\bibitem[Gir93]{MR1246390}
Emmanuel Giroux.
\newblock Topologie de contact en dimension {$3$} (autour des travaux de
  {Y}akov {E}liashberg).
\newblock {\em Ast\'erisque}, (216):Exp.\ No.\ 760, 3, 7--33, 1993.
\newblock S{\'e}minaire Bourbaki, Vol. 1992/93.

\bibitem[Gir00]{MR1779622}
Emmanuel Giroux.
\newblock Structures de contact en dimension trois et bifurcations des
  feuilletages de surfaces.
\newblock {\em Invent. Math.}, 141(3):615--689, 2000.

\bibitem[HKM07]{MR2318562}
Ko~Honda, William~H. Kazez, and Gordana Mati{\'c}.
\newblock Right-veering diffeomorphisms of compact surfaces with boundary.
\newblock {\em Invent. Math.}, 169(2):427--449, 2007.

\bibitem[Hon00a]{Honda1}
Ko~Honda.
\newblock On the classification of tight contact structures. {I}.
\newblock {\em Geom. Topol.}, 4:309--368, 2000.

\bibitem[Hon00b]{Honda2}
Ko~Honda.
\newblock On the classification of tight contact structures. {II}.
\newblock {\em J. Differential Geom.}, 55(1):83--143, 2000.

\bibitem[Kan98]{Kanda}
Yutaka Kanda.
\newblock On the {T}hurston-{B}ennequin invariant of {L}egendrian knots and
  nonexactness of {B}ennequin's inequality.
\newblock {\em Invent. Math.}, 133(2):227--242, 1998.

\bibitem[Kau89]{Kauffman}
Louis~H. Kauffman.
\newblock Invariants of graphs in three-space.
\newblock {\em Trans. Amer. Math. Soc.}, 311(2):697--710, 1989.

\bibitem[LCO16]{LO-Planar}
Peter Lambert-Cole and Danielle O'Donnol.
\newblock Planar {L}egendrian graphs.
\newblock 2016.

\bibitem[Mas69]{Mason}
W.~K. Mason.
\newblock Homeomorphic continuous curves in {$2$}-space are isotopic in
  {$3$}-space.
\newblock {\em Trans. Amer. Math. Soc.}, 142:269--290, 1969.

\bibitem[Ng03]{Ng}
Lenhard~L. Ng.
\newblock Computable {L}egendrian invariants.
\newblock {\em Topology}, 42(1):55--82, 2003.

\bibitem[OP12]{OP-1}
Danielle O'Donnol and Elena Pavelescu.
\newblock On {L}egendrian graphs.
\newblock {\em Algebr. Geom. Topol.}, 12(3):1273--1299, 2012.

\bibitem[OP14]{OP-Theta}
Danielle O'Donnol and Elena Pavelescu.
\newblock Legendrian {$\theta$}-graphs.
\newblock {\em Pacific J. Math.}, 270(1):191--210, 2014.

\bibitem[OP16]{OP-3}
Danielle O'Donnol and Elena Pavelescu.
\newblock The total {T}hurston-{B}ennequin number of complete and complete
  bipartite {L}egendrian graphs.
\newblock In {\em Advances in the Mathematical Sciences}, Associaton for Women
  in Mathematics. Springer, 2016.
\newblock to appear.

\bibitem[Tan15]{Tanaka}
Toshifumi Tanaka.
\newblock On the maximal {T}hurston-{B}ennequin number of knots and links in
  spatial graphs.
\newblock {\em Topology Appl.}, 180:132--141, 2015.

\bibitem[Whi32]{Whitney}
Hassler Whitney.
\newblock Congruent {G}raphs and the {C}onnectivity of {G}raphs.
\newblock {\em Amer. J. Math.}, 54(1):150--168, 1932.

\end{thebibliography}

\end{document}